\documentclass[10pt, reqno]{amsart}
\pdfoutput=1 
\usepackage{amsmath, amssymb, amsthm, enumerate, bbm, soul}
\usepackage[colorlinks, linkcolor=blue, citecolor=magenta, hypertexnames=false,backref]{hyperref} 
\usepackage[dvipsnames]{xcolor}

\usepackage{graphicx}
\usepackage{booktabs}
\usepackage[font=small]{caption} 
\usepackage[flushleft]{threeparttable}
\usepackage{mathtools, enumitem}
\usepackage{url}

\def\spine{1.1in}
\usepackage[
heightrounded,
top=1in,
bottom=1.3in,
inner=\spine,
outer=\spine 
]{geometry} %

\usepackage[final, expansion=true, protrusion=true, spacing=true, kerning=true]{microtype}
\microtypecontext{spacing= nonfrench}

\def\R{\mathbb R}

\def\S{\mathbb S}

\def\supp{\operatorname{supp}}
\def\theta{\vartheta}

\def\argmin{\operatorname{argmin}}
\def\vol{\textnormal{vol}}
\def\M{\mathcal{M}}
\def\ve{\varepsilon}

\numberwithin{equation}{section}

\newtheorem{theorem}{Theorem}[section]
\newtheorem{lemma}[theorem]{Lemma}

\newtheorem{proposition}[theorem]{Proposition}
\newtheorem{corollary}[theorem]{Corollary}
\newtheorem{definition}[theorem]{Definition}
\theoremstyle{definition}

\newtheorem{remark}{Remark}
\newcommand{\ra}{\rightarrow}

	\title{Energy, Polarization, and Separation of Greedy Sequences for Riesz and Green Kernels}

\author{Dmitriy Bilyk}
\address{School of Mathematics, University of Minnesota, Minneapolis, MN 55455, USA}
\email{dbilyk@umn.edu}

\author{Liudmyla Kryvonos}
\address{Department of Mathematics \& Statistics, 
	University of North Florida, Jacksonville, FL 32224, USA}
\email{liudmyla.kryvonos@unf.edu}

\author{Ryan W. Matzke}
\address{Department of Mathematics, Applied Mathematics, and Statistics, 
	Case Western Reserve University, Cleveland, OH 44106, USA}
\email{rwm106@case.edu}

\author{Edward Saff}
\address{Center for Constructive Approximation, Department of Mathematics, 
	Vanderbilt University, Nashville, TN 37240, USA}
\email{edward.b.saff@vanderbilt.edu}

\begin{document} 
\maketitle

\begin{abstract}
	We investigate the asymptotic behavior of greedy $s$-Riesz and Green energy sequences $\{x_{n}\}_{n=1}^{\infty}$ on the unit sphere $\S^{d} \subset \R^{d+1}$, where each point $x_n$ is defined as the minimizer of the discrete potential generated by the preceding points $x_1, x_2, ..., x_{n-1}$. We show that the greedy sequence attains optimal growth behavior for the second-order term of the Green and Riesz $s$-energies when $d-2 \leq s < d$. The main idea is to establish the bounds on polarization using well-separation properties of the greedy configurations.      
	
\end{abstract}

\section{Introduction} \label{sec:introduction}
Denote by $\omega_{N,d} := \{x_1, x_2,...,x_N\}$ a set of $N$ points on the unit sphere $\S^d := \{x \in \R^{d+1} : \|x\|=1\}$, where $\|\cdot\|$ denotes the Euclidean norm. Unless noted otherwise, we assume $d \geq 2$. For a symmetric, lower semi-continuous kernel $K: \S^d \times \S^d \rightarrow (-\infty, \infty]$ and a point set $\omega_{N,d}$, the discrete energy is defined by
$$
E_K(\omega_{N,d}) := \sum_{i \ne j}^{N} K(x_i,x_j).
$$
For a fixed $N$, we define the minimal energy by
$$
\mathcal{E}_{K}(N) := \underset{\omega_{N,d} \subset \S^d}{\min} E_K(\omega_{N,d}).
$$
We call any point configuration that attains $\mathcal{E}_{K}(N)$ \textit{optimal} and denote it by $\omega_{N,d}^{*}=\{x_1^{*}, x_2^{*},..., x_{N}^{*}\}$.

Other quantities closely related to energy are \textit{polarization} and \textit{maximal (constrained) polarization}, defined, respectively, as 
\begin{equation} \label{polarization}
P_{K}(\omega_{N,d}) = \underset{x \in \S^d}{\min} \sum_{j=1}^{N} K(x,x_j), \quad \textup{and} \quad \mathcal{P}_K(N) = \max_{\omega_{N,d} \subset \mathbb{S}^d} P_K(\omega_{N,d}). 
\end{equation}
We are particularly interested in the case of the Riesz $s$-kernels, given by \footnote {Many works in the literature omit the $1/s$ factor in the Riesz kernel, but we include it in our new as well as known cited results.}
\begin{equation} \label{def:kernel}
K_{s}(x,y) := \begin{cases} 
	\frac{1}{s} \|x-y\|^{-s}, \;\; s \ne 0 \\
	- \log\|x-y\|, \;\;\; s=0,
\end{cases}
\end{equation}
with $\log$ denoting the natural logarithm.
Note that on the sphere, we may write the Riesz kernels as functions of inner products
\begin{equation}\label{eq:Riesz on Sphere as inner product}
K_s(x,y) = \begin{cases}
\frac{1}{s} (2 - 2  \langle y, x \rangle)^{- \frac{s}{2}} & s \neq 0 \\
- \frac{1}{2} \log( 2 - 2  \langle y, x \rangle) & s = 0
\end{cases}.
\end{equation}
For brevity, we replace $K_s$ by $s$ when denoting energy and polarization, e.g. we denote Riesz energy by $\mathcal{E}_{s}$.

Both energy (see e.g. \cite{BelMO, BetS, BraG, GraS, HarS}) and polarization (see e.g. \cite{BDHSS22, FarN, FarR, Oht, RezSVl, RezSVo, Sim}), especially in the case of  Riesz kernels, have been actively studied over the past few decades: we  refer the reader to \cite{BHS} for a comprehensive exposition.
The second-order asymptotic behavior for the optimal Riesz $s$-energy is well understood (see the discussion in Section \ref{sec:RieszEnergyofGreedySeq}) for $-2<s<d$; the case $s \leq - 2$ is also settled, as a minimizer in this range consists of two antipodal points \cite{Bjork}. However, results for second-order asymptotics of greedy Riesz $s$-energy have not been previously obtained (except when $d=1$). For the Green and $s$-Riesz greedy sequences, $d-2 \leq s < d$, we establish upper bounds on their polarization, and then use them to show the optimality of their second-order terms. Also, we show that for $d=2$, a greedy sequence attains optimal second-order asymptotic behavior (up to constants) for the logarithmic energy.

\subsection{Construction of a Greedy Sequence}

We will mainly work with an infinite sequence $\{x_n\}_{n=1}^{\infty}$ instead of a sequence of configurations. This is more difficult since points, once placed, remain fixed.  We start with an arbitrary initial point $x_1 \in \mathbb{S}^d$. The construction is then greedy: at each step we add a point that minimizes the $K$ energy with respect to the existing set of points.  Namely, having determined $\omega_{N,d}$, we pick
\begin{equation}\label{eq:greedy def}
 x_{N+1} = \argmin_{x \in \mathbb{S}^d} \sum_{i=1}^{N} K(x, x_i) = \argmin_{x\in \mathbb{S}^d} E_{K} \big( \omega_{N,d}  \cup \{x\}\big). 
 \end{equation}
If the minimum is not assumed at a unique point, then any such point may be chosen. 
 Notice that $x_{N+1}$ is just the point where the polarization is achieved for the set $\{x_1,\dots, x_N\}$, i.e.  $$  \sum_{j=1}^{N} K(x_{N+1}, x_j)  = P_{K} \big( \{x_1,\dots, x_N\} \big).  $$
Observe that the energy and polarization of greedy sequences satisfy the following relation.
\begin{equation} \label{eq:Gr Energy POlarization relation}
E_{K}(\omega_{N,d}) = 2 \sum_{ k = 1}^{N-1} P_{K}(\omega_{k,d}).
\end{equation}

When $K=K_0$,  i.e. $K_0(x,y) = - \log\|x-y\|$, {and $d = 1$ (or more generally, on any compact $\Omega \subset \mathbb{C}$)}, such sequences are called \textit{Leja sequences}  in recognition of the work of Leja \cite{Lej57}, although they were earlier introduced by Edrei \cite{Edr40}. Simultaneous with the former work, G\'{o}rski studied the case of $s = 1$ (the Newtonian kernel) for compact sets $\Omega \subset \mathbb{R}^3$ \cite{Gor57}. For $s = d-2$ (i.e. the Green kernel of $\mathbb{R}^d$) and compact $\Omega \subset \mathbb{R}^d$, the corresponding greedy sequences are known as Leja--G\'{o}rski sequences, which have been studied in \cite{Got01, Pri11}. Leja points have applications in Stochastic Analysis \cite{NarJ14} and Approximation/Interpolation Theory \cite{BiaCC12, CalVM11, CalVM12, Chk13, DeM04, JanWZ19, Rei90, TayT10}, and various numerical methods have been developed to approximate  such points (see, e.g. \cite{BagCR98, BosMSV10, CorD01}).

\subsection{Well-separation of Greedy Sequences}\label{subsec:Well-separated}

Let 
\begin{equation}\label{eq:def of separation distance}
    \delta(\omega_{N,d}) :=\{\min \|x_i - x_j\|: \: x_i, x_j \in \omega_{N,d},  \; j\ne i\}
\end{equation}
denote the \textit{separation distance} of the set $\omega_{N,d}$.

In his pioneering work \cite{Dah78}, Dahlberg proved that for a $C^{1, \alpha}$-smooth $d-$dimensional manifold $M \subset \R^{d+1}$ without boundary and $s=d-1$, there exists a constant $c>0$ such that
\begin{equation} \label{eq:Dalhberg separation bound}
\delta(\omega^{*}_{N,d}) \geq \frac{c}{N^{1/d}} ,\;\;\; \forall N \in \mathbb{N}.
\end{equation}
It is easy to see that the asymptotic order in (\ref{eq:Dalhberg separation bound}) cannot be improved. Based on this, we will refer to a sequence of point configurations $\{\omega_{N,m}\}_{N \in \mathbb{N}}$ on a $m-$dimensional submanifold of $\R^n$ as \textit{well-separated} if there exists a constant $c>0$ such that $\delta(\omega_{N,m}) \geq c N^{-1/m}$ for all $N \in \mathbb{N}$.

Rakhmanov, Saff, and Zhou established in \cite{RakSZ} well-separation properties of the minimal logarithmic energy points on $\S^{2}$. We remark that their proof can be easily applied without any changes to the greedy points as well. For $d-1 \leq s< d$, the separation properties for the optimal (minimal energy) configurations on $\S^{d}$ was proved by Kuijlaars, Saff, and Sun \cite{KuiSS}. The range of $s$ was later extended to $d-2 \leq s < d$ by Brauchart, Dragnev, and Saff \cite{BraDS14} with explicit values for the constant $c$. By following their techniques, we are able to extend the result to greedy sequences and obtain the following theorem, presented in Section \ref{sec:separation}.
\begin{theorem} \label{thm:sub_separation}
For $d \geq 2$, $d-2 \leq s<d$, sequences of greedy $s$-energy configurations on $\S^d$ are well-separated. 
\end{theorem}

Hardin, Reznikov, Saff, and Volberg in \cite{HarRSV19} proved well-separation of both optimal and greedy configurations for smooth $d$-dimensional manifolds without boundary in $\R^{d+1}$ and $s \in [d-1,d)$. In the same work, they establish the property for the hypersingular case, $s>d$, for the sets with positive $d$-dimensional Hausdorff measure, thereby extending the known result of Kuijlaars and Saff \cite{KuiS98} on the optimal configurations on the sphere $\S^{d}$.    
While it still remains open whether minimal logarithmic and $s$-Riesz energy configurations with $0<s< d-2$ are well-separated on $\S^{d}$, $d \geq 3$, Damelin and Maymeskul \cite{DamM05} obtained separation of order $N^{-1/(s+1)}$ for $0<s < d-1$. The following theorem extends their result to greedy sequences; see the proof provided in Section \ref{sec:separation}.

\begin{theorem} \label{th:separation}
For  $0<s < d-2$, any sequence of greedy $s$-energy configurations $\{\omega_{N,d}\}_{N \in \mathbb{N}}$ on $\S^{d}$ satisfies
$$
\delta(\omega_{N,d}) \geq c N^{-1/(s+1)}, \;\;\; \forall N \in \mathbb{N}, 
$$
with the constant $c >0$ independent of $N$.
\end{theorem}

In Theorem \ref{thm:improved separation superharmonic}, we establish an improved separation order of $N^{-1/(s+2)}$, achievable under certain bounds on the polarization.

\subsection{Optimal Second-order Polarization Estimates}  
For $s<d$, we note that the kernel $K_s(x,y)$ is integrable on $\mathbb S^d$ in each variable and depends only on the distance between $x$ and $y$. We denote the $s$-energy of $\S^{d}$ by
\begin{equation}\label{eq:Continuous Min Energy}
 I_{s} = {I}_{s,d} := \int_{\mathbb{S}^d} \int_{\mathbb{S}^d} K_s (x,y) d \sigma(x) d\sigma(y) = \int_{\mathbb{S}^d} K_s(x,y) d\sigma(y) , \quad \forall~x \in \mathbb{S}^{d}, \quad
\end{equation}
where $\sigma$ denotes the normalized uniform measure on $\mathbb{S}^d$. Throughout the paper, we use the notation $I_s$ interchangeably with $I_{s,d}$ whenever the dimension $d$ is clear from the context. The constant $I_s$ is also called the Wiener constant for $s$-energy on $\S^{d}$.

One would expect that polarization is maximized when the points are distributed as uniformly  as possible. In that case, one could replace summation by integration in \eqref{polarization} and expect the maximal Riesz polarization $\mathcal{P}_{s}(N)$ to behave roughly like $I_{s}  N$, and this is indeed correct, see e.g. \cite[Theorem 14.6.3]{BHS} together with  \cite[Corollary 14.6.7]{BHS}. For the well-separated configurations, we prove more delicate upper bounds on polarization.

\begin{theorem}\label{thm:Riesz Polarization}
Suppose a sequence of configurations $\{\omega_{N,d}\} \subset \S^{d}$ is well-separated. Then there exist positive constants $c_{s,d}$, depending on the separation constant, such that the following hold for all $N \in \mathbb{N}$. If $d \geq 2$ and $d-2 \leq s < d$, then
\begin{equation}\label{eq:opt_polarization}
  P_{s}(\omega_{N,d})   \leq I_{s,d} N - c_{s,d} N^{\frac{s}{d}}
\end{equation}
except when $d=2$ and $s=0$, in which case
\begin{equation}\label{eq:opt_polarization2}
   P_{0}(\omega_{N,2})   \leq  I_{0,2} N  + c_{0,2} \Big(-\frac{1}{2} + D(N) \Big) \log N,
\end{equation}
where $D(N)$ is defined by 
\begin{equation} \label{def: of Dn}
E_{0}(\omega_{N,2}) = I_{0,2}(\sigma) N^2 - D(N)N\log N.
\end{equation}
In particular, (\ref{eq:opt_polarization}) and (\ref{eq:opt_polarization2}) hold for greedy $s$-energy sequences when $s$ is in the prescribed range. Although the behavior of $D(N)$ for arbitrary well-separated configurations is unknown, in Theorem \ref{thm:greedyenergy} we show that for greedy sequences $C_1 \leq D(N) \leq \frac{1}{2} + \frac{C_2}{\log N}$ for some positive constants $C_1, C_2$, so that (\ref{eq:opt_polarization2}) becomes

\begin{equation}
   P_{0}(\omega_{N,2})   \leq  I_{0,2} N  + c_{0,2}.
\end{equation}
\end{theorem}

The upper bounds (\ref{eq:opt_polarization}), (\ref{eq:opt_polarization2}) in the Theorem \ref{thm:Riesz Polarization} are proved in Sections \ref{thm:sec:ProofPolarizRiesz} and \ref{sec:ProofPolarizLog}, respectively.


\subsection{Energy Asymptotics of Greedy Sequences}  \label{sec:RieszEnergyofGreedySeq}
Our other main result establishes the second-order asymptotic behavior of greedy sequences on $\S^d$ for $d\geq 2$. We show that the second-order term for greedy sequences has the optimal order of magnitude, matching that of the optimal configurations.

\begin{theorem}
\label{thm:greedyenergy}
Let $d \geq 2$, and denote by $\omega_{N,d} = \{x_1,...,x_N\}$ the first $N$ elements of a greedy $s$-energy sequence on $\S^d$ as defined in (\ref{eq:greedy def}) for $K=K_s$. Then there exist positive constants $c_{s,d}, C_{s,d}$  such that the following holds for $N \in \mathbb{N}\setminus\{1\}$.\\
(i) \noindent For $d-2 \leq s < d$, $s > 0$,
\begin{equation}\label{eq:greedyenergy}
 - c_{s,d} N^{1 +\frac{s}{d}} \leq E_{s}(\omega_{N,d})  - I_{s,d} N^2  \leq  - C_{s,d} N^{1+\frac{s}{d}}.
\end{equation}
(ii) For $s = 0$, $d=2$,
\begin{equation}\label{eq:greedyenergy2}
- \frac{1}{2} N \log N -  c_{0,2} N \leq  E_{0}(\omega_{N,2})  - I_{0,2} N^2  \leq - C_{0,2} N \log N.
\end{equation}
\end{theorem}
The case of $s \leq -2$ was entirely handled in \cite[Theorems 4.1 and 5.2]{LopM21} and the case of $s \geq d$ has been studied in  \cite{LopS10, LopW15}.
In the case of the circle $\mathbb S^1$, i.e. $d=1$, more precise information about the Riesz energy of greedy sequences is known, see Theorem \ref{thm:Energy of Greedy points Circle}.

We prove the upper bounds (i) and (ii) in Theorem \ref{thm:greedyenergy} in Sections \ref{sec:ProofofGreedyEnergyRiesz} and \ref{sec:ProofGreedyEnergyLog}, respectively. The lower bounds are well-known optimal estimates for minimal Riesz $s$-energies in the range $-2<s<d$ that have been computed by various authors \cite{Bra06, BraHS12, KuiS98, RakSZ94, Wag90, Wag92}, see also Theorems 6.4.5, 6.4.6, and 6.4.7 in \cite{BHS}. These bounds are as follows. 

\begin{theorem}\label{thm:Riesz Energy Asymptotics}
For $d \geq 1$, $-2 < s < d$, $s \neq 0$, there exist  positive constants $C_{s,d}$, $C'_{s,d}$ such that for $N \geq 2$,
\begin{align*} \label{thm:BoundsonOptimalRieszEnergy}
I_{s,d} N^2 - C'_{s,d} N^{1+\frac{s}{d}} & \le \mathcal{E}_{s}(N) \le 
I_{s,d} N^2 - C_{s,d} N^{1+\frac{s}{d}}, \quad  s>0,\\
I_{s,d} N^2 + C'_{s,d} N^{1+\frac{s}{d}} & \le \mathcal{E}_{s}(N) \le 
I_{s,d} N^2 +  C_{s,d} N^{1+\frac{s}{d}}, \quad  s<0. 
\end{align*} 
If $s = 0$, then
\begin{equation} \label{thm:BoundsonOptimalLogEnergy}
\mathcal{E}_{0}(N) = I_{0,d} N^2 - \frac{1}{d}N \log N + \mathcal{O}(N).
\end{equation}
\end{theorem}

\subsection{Green Energy}
Recall that Green's function $G(x,y) = G(\mathcal{M}; x,y)$ for the Laplacian on a compact Riemannian manifold $(\mathcal{M}, g)$ of dimension $d \geq 2$ is symmetric, smooth off the diagonal, and satisfies
 $$
 \bigtriangleup G(x, \cdot) = \delta_x - V^{-1} \textnormal{vol}
 $$
in the sense of distributions, where $\bigtriangleup = - \operatorname{div} \nabla$ is the Laplace-Beltrami operator on $\mathcal{M}$, $\textnormal{vol}$ is the Riemannian volume form or density, and $V = \textnormal{vol}(\mathcal{M})$ is the volume of $\mathcal{M}$. See \cite[Theorem 4.13]{Aub98} for a proof of the existence of such a function, which is discussed further in \cite{BelCC19, Del19}. 

Given a configuration $\omega_N =\{x_1, x_2,...,x_N\} \subset \mathcal{M}$, its Green energy is 
$$
E_{G}(\omega_N) = \sum_{i \ne j} G(\mathcal{M}; x_i,x_j),
$$
and has been studied, for example, by Elkies \cite[VI, Theorem 5.1]{Lang}, Beltr\'an, Corral, and Criado del Rey \cite{BelCC19}.

Our primary focus is on comparing the second-order asymptotic behavior of greedy Green energy configurations with that of optimal ones, for which the following is known. 
\begin{theorem}
    On the sphere $\mathbb{S}^d$, for $d \geq 3$, for $N$ sufficiently large, there exist positive constants $C_1, C_2$ such that
    \begin{equation}\label{eq: Green energy second order asymptotics}
        - C_1 N^{1 + \frac{d-2}{d}} \leq \mathcal{E}_G(N) \leq - C_2 N^{1 + \frac{d-2}{d}}. 
    \end{equation}
\end{theorem}

\begin{remark}
 We note that, similarly to the Riesz $s$-energy, the first order term of the discrete Green energy $\mathcal{E}_G(N)$ is of the form $I_G N^2$, where $I_G:= \int_{\mathbb{S}^d} \int_{\mathbb{S}^d} G (x,y) d \sigma(x) d\sigma(y)$. However, due to the properties of the Green function, this integral vanishes, which is why it doesn't appear in equation (\ref{eq: Green energy second order asymptotics}).
\end{remark}

The lower bound was proven for general compact Riemannian manifolds in \cite[pg. 4, Corollary]{Ste21}, and proven for spheres, with an explicit constant, in \cite[Theorem 4.2]{BelL23}. For the other compact, connected, two-point homogeneous spaces, this lower bound (with explicit constants) was proven in \cite[Theorem 1.1]{BelTL24} and \cite[Theorem 4.14]{AndDGMS23}, and the upper bound was proven in \cite[Corollary 4.3]{AndDGMS23}. There is not currently an explicit proof of the upper bound for the minimal Green energy on the sphere, though the community has understood it to be as above for years. For completeness, we include a proof in Appendix \ref{sec:Upper bound on Green energy}.


Note that on $\S^2$, the Green function coincides, up to multiplicative and additive constants, with the logarithmic potential, namely 
$$
G(\S^2;p,q) = - \frac{1}{2\pi} \log\|p-q\| - \frac{1}{4 \pi} + \frac{\log 2}{2 \pi},
$$
which is the case treated in Section \ref{sec:RieszEnergyofGreedySeq}.

In Section \ref{sec:Green energy}, as in the Riesz and logarithmic cases, we first establish the well-separation of greedy Green energy configurations and then derive bounds for their polarization and energy.

For a compact connected Riemannian manifold $\mathcal{M}$ with geodesic distance $d_{\mathcal{M}}(x,y)$ and Green kernel $G$, let $\delta_{\mathcal{M}}(\omega_N) : = \min \{ d_{\mathcal{M}} (x_i, x_j): \: x_i, x_j \in \omega_N, j\ne i\}$ be the separation distance of a set $\{x_1, ..., x_N\} = \omega_N \subset \mathcal{M}$.

\begin{theorem}\label{thm:Separation bound for Greedy green points}
Any sequence $\{\omega_N\}_{N \in \mathbb{N}}$ of greedy points for the discrete Green energy on a compact connected Riemannian manifold $\mathcal{M}$ of dimension $d\geq 2$ is well-separated, i.e.
$$
\delta_{\mathcal{M}}(\omega_N) \geq \frac{C}{N^{1/d}}, \;\;\;  \forall N \in \mathbb{N}, 
$$
where $C>0$ is a constant independent of $N$.
\end{theorem} 

\begin{theorem}\label{thm:GreenPolarizationAsymptotics}
Suppose $d \geq 3$ and the sequence of configurations $\{\omega_{N,d}\} \subset \S^{d}$ is well-separated. Then there exist positive constants $c_{G,d}, b_{G,d}$ such that the following hold for all $N \in \mathbb{N}$. If $d \geq 3$,
\begin{equation} \label{eq: GreenPolarizationUpperBound}
    P_G(\omega_{N,d}) \leq  c_{G,d} \Big( D(N) - b_{G,d}  \Big) N^{1-\frac{2}{d}}
\end{equation}
where $D(N)$ is defined by 
$$ E_{G}(\omega_{N,d}) =  - D(N)N^{2-2/d}. $$

In particular, (\ref{eq: GreenPolarizationUpperBound}) holds for greedy Green energy sequences. Although the behavior of $D(N)$ for arbitrary well-separated configurations is unknown, in Theorem \ref{thm: Energy for Green kernel} we show that for greedy sequences $C_1 \leq D(N) \leq C_2$ for some positive constants $C_1, C_2$.
\end{theorem}

The proof of this result is found in Section \ref{sec:GreenPolarization}.

\begin{theorem} \label{thm: Energy for Green kernel}
For $d \geq 3$, and $G$ the Green kernel for $\S^d$, if $\omega_{G,N}$ denotes the first $N$ elements of a greedy $G$-energy sequence on $\mathcal{M}$, there exist positive constants $C_{G,1}, C_{G,2}$ such that
\begin{equation}\label{eq:Greedy Green Asymptotic}
- C_{G,1} N^{2 - 2/d} \leq E_G(\omega_N) \leq - C_{G,2} N^{2 - 2/d}
\end{equation}

\end{theorem}

The upper bound for \eqref{eq:Greedy Green Asymptotic} is proved in Section \ref{sec:Asymptotics for Green energy Greedy}, whereas the lower bound follows from the bound for minimal energy, found in \cite[Theorem 4.2]{BelL23}.

The paper is organized as follows. In Sections \ref{sec:Proof Polarization} and \ref{sec:EnergyAsymptotics}, we prove Theorems \ref{thm:Riesz Polarization} and \ref{thm:greedyenergy}, respectively. In Section \ref{sec:Green energy} we examine the energy, polarization, and separation properties of greedy Green sequences. Section~\ref{sec:separation} presents the proofs of the separation properties of greedy \( s \)-Riesz points. Finally, in Section \ref{sec:polarization for s<d-2}, we establish polarization bounds for the \( s \)-Riesz kernel in the case \( s < d-2 \).

\section{Proof of Theorem \ref{thm:Riesz Polarization}}\label{sec:Proof Polarization}

Suppose $s \in [0, d)$, such that for all $y \in \mathbb{S}^d$, as a function of $x$, $K_s(x,y)$, given by (\ref{def:kernel}), is subharmonic on $\mathbb{S}^d \setminus \{y\}$. Here, subharmonicity is defined via the Laplace-Beltrami operator, or alternatively, that the average over a geodesic ball on $\mathbb{S}^d$ is greater than or equal to the value at the center of such a ball. Note that on the sphere, the Laplace-Beltrami operator $\bigtriangleup$ can be written using standard spherical coordinates $\theta_1, ..., \theta_{d}$ (see, e.g., \cite[equation (2.2.4)]{KozMR01})
\begin{equation}\label{eq:Laplace-Beltrami on Sphere}
\bigtriangleup =  -\sum_{j=1}^{d} \frac{1}{q_j (\sin(\theta_j))^{d-j}} \frac{\partial}{\partial \theta_j} \Big( \big(\sin(\theta_j) \big)^{d-j} \frac{\partial}{\partial \theta_j} \Big),
\end{equation}
where $q_1 = 1$ and $q_j = \prod_{k=1}^{j-1} \sin^2(\theta_k)$.

For a fixed $y \in \mathbb{S}^d$, we may choose the coordinates so that $\cos(\theta_1) = \langle y, x \rangle$, so, using \eqref{eq:Riesz on Sphere as inner product},
\begin{equation}\label{eq:Laplace Beltrami of Riesz on Sphere}
\bigtriangleup_{x} K_s(x,y) = -\frac{1}{2} \|x - y\|^{-s-2} \Big( s +2 + (s+2-2d) \langle x, y \rangle \Big),
\end{equation}
which is nonpositive for $x \in \mathbb{S}^d \setminus\{y\}$ when $d-2 \leq s$. For $s=0$, we have

\begin{equation}
\bigtriangleup_{x} K_0(x,y) =  \frac{ (d-1) \langle x, y \rangle -1 }{\|x-y\|^2},
\end{equation}
so that $\bigtriangleup_{x} K_0(x,y) = -1/2$ for $d=2$.

We next state the mean value property of subharmonic functions on $\S^d$, which is central to our arguments.
\begin{lemma}(S. Gardiner, \cite{Gar})
	Suppose $\bigtriangleup u(x) \leq 0$ on an open subset $U$ of the unit sphere $\mathbb{S}^d$. Then mean value inequality holds for $u$ on $U$, i.e. for any $B(p, r_0) \subset U$ with $r_0 < \pi$,
	\begin{equation} \label{eq:meanv}
	u(p) \leq \frac{1}{\sigma(B(p, r_0))} \int_{B(p, r_0)} u(x) d \sigma(x) .
	\end{equation}
		
\end{lemma}
\begin{proof}
	Let $C[U]$ be a conical set in $R^{d+1}$ consisting of all rays from $0$ through $U$.
	Define
	$$
	u^{*}(x):= u\left(\frac{x}{\|x\|}\right) 
	$$
	the extension of $u$ from $U$ to the set $C[U]$.
	
	Applying the properties of the harmonic measure $\omega(p, C[B(p, r)])$ of $C[B(p, r)]$ detailed in \cite[Theorem 3.11]{HayK76}, for any $B(p, r) \subset U$ we obtain 
	$$
	u(p) = u^{*}(p)  \leq  \int_{\partial C[B(p, r)]}     u^{*}(x) d\omega(p, C[B(p, r)])(x), 
	$$
 since the point at infinity carries zero harmonic measure for such conical sets. 
	
	Since $u^{*}$ is constant along each ray, and the measure $\omega(p, C[B(p, r)])$ is axially symmetric, we can write
    $$
    \int_{\partial C[B(p, r)]}     u^{*}(x) d\omega(p, C[B(p, r)])(x) =  \frac{1}{\sigma_{d-1}(S(p,r))} \int_{S(p,r)} u(x) d\sigma_{d-1}(x),
	$$
	with $S(p,r) := \partial C[B(p, r)] \cap \mathbb{S}^{d}$, and $\sigma_{d-1}$ being surface area measure for the $d-1$ dimensional sphere $S(p,r)$.
	
	Finally, we integrate
	$$
 \int_{0}^{r_0}	\sigma_{d-1}(S(p,r)) u(p)  dr\leq  \int_{0}^{r_0}\int_{S(p,r)} u(x) d\sigma_{d-1}(x) dr
	$$
	to get (\ref{eq:meanv}).
 \end{proof}

\subsection{Proof of Theorem \ref{thm:Riesz Polarization} for $d-2 \leq s < d$, $d \geq 2$, $s>0$.} \label{thm:sec:ProofPolarizRiesz}

For each $N \in \mathbb{N}$, let $\{ x_{N,1}, ..., x_{N,N}\} = \omega_{N,d}$.

Since our sequence of point configurations, $\{\omega_{N,d}\}$, is well-separated, as defined in Section \ref{subsec:Well-separated}, we know that there is some constant $c > 0$ such that for each $N \in \mathbb{N}$
\begin{equation}
\min_{1 \leq i < j \leq N} \|x_{N,i} - x_{N,j}\| \geq \frac{c}{N^{1/d}}. 
\end{equation}

This means that
\begin{equation}
\min_{1 \leq i < j \leq N} \arccos( \langle x_{N,i},  x_{N,j} \rangle) \geq \arccos\bigg( 1 - \frac{c^2}{2 N^{2/d}} \bigg).
\end{equation}

Let $B(x, r)$ be the open geodesic ball of radius $r$ centered at $x$, and $r_N = \arccos \Big( 1 - \frac{c^2}{8 N^{2/d}}\Big)$ for each $N \in \mathbb{N}$. For $N \in \mathbb{N}$, let us define the probability measure $\mu_N$ by
\begin{equation} \label{def:measure muN}
d \mu_N(x) := \frac{1}{1 - N \sigma ( B(p, r_N))} \Big( 1 - \sum_{j=1}^{N} \mathbbm{1}_{B( x_{N,j}, r_N)}(x) \Big) d \sigma(x),
\end{equation}
where $\sigma$ is the normalized uniform measure on the sphere. Note that due to the separation, this is indeed a probability measure.
Then we have
\begin{align*}
P_{s}(\omega_{N,d}) & := \min_{x \in \mathbb{S}^d} \sum_{j=1}^{N} K_s(x, x_{N,j}) \\
 & \leq \sum_{j=1}^{N}  \int_{\mathbb{S}^d} K_s(x, x_{N,j}) d\mu_N (x) \\
 & = \frac{1}{1 - N \sigma ( B(p, r_N))} \Big( N I_s( \sigma) - N \int_{B(p, r_N)} K_s(x, p) d \sigma(x)  - 2\sum_{1 \leq i < j \leq N} \int_{B(x_{N,i}, r_N)} K_s(x, x_{N,j}) d \sigma(x) \Big).
\end{align*}
Using the subharmonicity of our kernel and the fact that $\underset{ x \in \overline{B(p, r_N)}}{\argmin} \; K_s(x, p)$ is on the boundary of the ball
we have that
\begin{align*}
P_{s}(\omega_{N,d}) & \leq  \frac{1}{1 - N \sigma ( B(p, r_N))} \Big( N I_s( \sigma) - N \int_{B(p, r_N)} K_s(x, p) d \sigma(x) - \sigma(B(p,r_N)) 2\sum_{1 \leq i < j \leq N}  K_s(x_{N,i}, x_{N,j})  \Big)\\
& \leq  \frac{1}{1 - N \sigma ( B(p, r_N))} \Big( N I_s( \sigma)  - N \sigma(B(p,r_N)) \frac{1}{s} (2 - 2 \cos(r_N))^{-s/2} - \sigma(B(p,r_N)) E_{s}( \omega_{N,d}) \Big)\\
& = \frac{1}{1 - N \sigma ( B(p, r_N))} \Big( N I_s( \sigma)  - N \sigma(B(p,r_N)) \frac{2^{s}}{s c^s} N^{s/d} - \sigma(B(p,r_N)) E_{s}( \omega_{N,d}) \Big).
\end{align*}

Theorem \ref{thm:Riesz Energy Asymptotics} yields $I_{s} N^2 - C'_{s,d} N^{1+\frac{s}{d}} \leq \mathcal{E}_s(N) \leq  E_{s}( \omega_{N,d})$, and thus the polarization bound takes the form

\begin{equation}
P_{s}(\omega_{N,d}) \leq N I_s( \sigma)   +   \frac{\sigma(B(p,r_N)) N^{1+\frac{s}{d}}}{1 - N \sigma ( B(p, r_N)) } \Big(  -   \frac{2^{s}}{s c^s} + C'_{s,d}    \Big)
\end{equation}

Then, using the fact that for some constant $A>0$, $\sigma(B(p, r_N)) = \frac{A}{N} + \mathrm{o}(N^{-1})$, and assuming that the constant $c$ is sufficiently small to ensure that the term $(- \frac{2^{s}}{s c^s} + C'_{s,d})$ is negative, we obtain for some positive $C$, and all $N \geq 2$,
\begin{equation} \label{eqt:polarization bound}
P_{s}(\omega_{N,d}) \leq N I_s( \sigma) - C N^{\frac{s}{d}}.
\end{equation}

\qed 

\begin{remark}
The same technique can be applied to configurations with separation of order $N^{-1/\alpha}$, $\alpha < d$, by taking $r_N = \arccos\Big( 1 - \frac{c^2}{8 N^{2/\alpha}} \Big)$ in (\ref{def:measure muN}). In this case, the resulting upper bound for polarization estimates would be $N I_s - C N^{1+\frac{s-d}{\alpha}}$. We suspect that for maximal polarization the optimal behavior is  $\mathcal{P}_s(N) = N I_s + O(N^{s/d})$. 
\end{remark}

\subsection{Proof of Theorem \ref{thm:Riesz Polarization} for $s=0$, $d=2$.} \label{sec:ProofPolarizLog}

The proof essentially mirrors the case $s>0$. As before, for each $N \in \mathbb{N}$, let $\{ x_{N,1}, ..., x_{N,N}\} = \omega_{N,2}$. Due to well-separation of sequence of point configurations, $\{\omega_{N,2}\}$, there is some constant $c > 0$ such that for each $N \in \mathbb{N}$
\begin{equation}
\min_{1 \leq i < j \leq N} \|x_{N,i} - x_{N,j}\| \geq \frac{c}{N^{1/2}}. 
\end{equation}

Similar to (\ref{def:measure muN}), we define a probability measure $\mu_N$, where we take $r_N = \arccos \Big( 1 - \frac{c^2}{8 N} \Big)$ for each $N \in \mathbb{N}$.

We have that
\begin{align*}
P_{{0}}(\omega_{N,2}) & := \min_{x \in \mathbb{S}^2} \sum_{j=1}^{N} (- \log \|x -x_{N,j} \| )\\
 & \leq \sum_{j=1}^{N}  \int_{\mathbb{S}^2} (- \log \|x - x_{N,j} \| ) d\mu_N (x) \\
 & = \frac{1}{1 - N \sigma ( B(p, r_N))} \Big( N \int_{\mathbb{S}^2} (- \log \|x - y \|) d\sigma(x)  - \sum_{i,j = 1}^{N} \int_{B(x_{N,i}, r_N)} (- \log \|x - x_{N,j} \|) d \sigma(x) \Big) \\
  & = \frac{1}{1 - N \sigma ( B(p, r_N))} \Big( N I_0(\sigma) - N \int_{B(p, r_N)} (- \log \|x - p \| ) d \sigma(x) \\
  & \qquad \qquad \qquad \qquad \qquad - 2\sum_{1 \leq i < j \leq N} \int_{B(x_{N,i}, r_N)} (- \log \|x - x_{N,j} \| ) d \sigma(x) \Big)\\
  & \leq \frac{1}{1 - N \sigma ( B(p, r_N))} \Big( N I_0(\sigma) - N \sigma (B(p, r_N)) (- \log \frac{c}{2 \sqrt{N}}) \\
  & \qquad \qquad \qquad \qquad \qquad - 2\sum_{1 \leq i < j \leq N} \int_{B(x_{N,i}, r_N)} (- \log  \|x - x_{N,j} \|) d \sigma(x) \Big)\\ 
\end{align*}
Recalling the formula for the area of a spherical cap, we find
\begin{equation}
\sigma ( B(p, r_N))  = \frac{1- \cos(r_N)}{2} = \frac{c^2}{16 N}.
\end{equation}
Due to subharmonicity of the logarithmic kernel, we see that
\begin{align*}
& P_{{0}}(\omega_{N,2})  \leq \frac{1}{1 - N \sigma ( B(p, r_N))} \Bigg( N I_0(\sigma) + N \sigma (B(p, r_N)) \log \frac{c}{2\sqrt{N}}   \\
  & \qquad - 2 \sigma ( B(p, r_N)) \sum_{1 \leq i < j \leq N} (- \log \|x_{N,i} - x_{N,j} \|)  \Bigg)\\
& \leq \frac{1}{1 - \frac{c^2}{16}} \Bigg( N I_0(\sigma) + \frac{c^2}{16}  \log \frac{c}{2\sqrt{N}}  -  \frac{c^2}{16 N} E_{0}(\omega_{N,2}) \Bigg)\\
\end{align*}
Thus, if we let $D(N)$ to be defined as 
$$
E_{0}(\omega_{N,2}) = I_0(\sigma) N^2 - D(N)N\log N,
$$
then we have
\begin{equation}
 P_{{0}}(\omega_{N,2})  \leq  I_{0}(\sigma) N +  \frac{c^2}{16 -c^2} \bigg( - \frac{1}{2} + D(N) \bigg),
\end{equation}
which completes the proof.
\qed

\section{Proof of Theorem \ref{thm:greedyenergy}}\label{sec:EnergyAsymptotics}

\subsection{Proof of (i), Theorem \ref{thm:greedyenergy}} \label{sec:ProofofGreedyEnergyRiesz}

For $d-2  \leq s < d$, $s > 0$, the claimed energy estimates follow immediately from (\ref{eq:Gr Energy POlarization relation}) and (\ref{eqt:polarization bound}). 
\begin{align*}
E_{{s}}(\omega_{N,d}) & = 2 \sum_{ k = 1}^{N-1} P_{s}(\omega_{k,d})\\
& \leq 2 \sum_{ k = 1}^{N-1} \Bigg( k I_s(\sigma) - C k^{\frac{s}{d}} \Bigg)\\
& = N(N-1) I_s(\sigma) - \frac{2d}{d+s} C N^{1+ \frac{s}{d}} + O(N^{\frac{s}{d}})\\
& = N^2 I_s(\sigma) - \frac{2d}{d+s} C N^{1+ \frac{s}{d}} + O(N),
\end{align*}

\qed

\subsection{Proof of (ii), Theorem \ref{thm:greedyenergy}}  \label{sec:ProofGreedyEnergyLog}

Due to the lack of information on the second-order term in estimate (\ref{eq:opt_polarization2}), this case requires a different approach. 

For each $N \geq 2$, let $D_1(N)$ and $D(N)$ be defined by
\begin{equation}
\mathcal{E}_{0}(N) = I_0(\sigma) N^2 - D_1(N) N \log N
\end{equation} 
and
\begin{equation}
E_{0}(\omega_N) = I_0(\sigma) N^2 - D(N) N \log N.
\end{equation}
From \cite[Theorem 2.1]{LopS10}, and \cite[Theorem 2.3]{Lop10}, we know that
\begin{equation}
\lim_{N \rightarrow \infty} \frac{E_{0} (\omega_{N,2})}{N^2} = I_0(\sigma).
\end{equation}
In addition, there exits a constant $C_{\log} >0$, such that
\begin{equation}\label{eq:Simple bounds on Greedy Log energy S2}
N^2 I_{0}(\sigma) - \frac{1}{2} N \log N - C_{\log} N  < \mathcal{E}_{0}(N) \leq E_{0}(\omega_N) \leq N^2 I_{0}(\sigma) - N I_{0}(\sigma),
\end{equation}
where the lower bound follows from (\ref{thm:BoundsonOptimalLogEnergy}), and the upper bound is obtained by applying the maximal polarization estimates from~\cite[Equation 14.6.4]{BHS} to (\ref{eq:Gr Energy POlarization relation}), where the constant \( T_0(\mathbb{S}^2) \) in the equation is, according to \cite[Corollary 14.6.7]{BHS}, precisely the Wiener constant \( I_0(\sigma)\). 
Therefore, we immediately see, due to \eqref{eq:Simple bounds on Greedy Log energy S2},
$$ \frac{1}{2} + \frac{C_{\log}}{\log N}  \geq D_1(N) \geq D(N) > 0.$$

 Notice that the energy for $\omega_{N,2}$
\begin{equation}
E_{0}(\omega_{N,2}) = 2 \sum_{ 1 \leq i < j \leq N} ( - \log \|x_i - x_j\| )
\end{equation}
can be written as
\begin{equation}
E_{{0}}(\omega_{N,2})= 2 \sum_{ k = 1}^{N-1} P_{0}(\omega_{k,2}).
\end{equation}

We have that
\begin{align*}
I_0(\sigma) (N+1)^2 - D(N+1) (N+1) & \log(N+1)  = E_{0}(\omega_{N+1,2})\\
& = E_{0}(\omega_{N,2})+ 2 P_{0}( \omega_{N,2}) \\
& \leq  I_0(\sigma) N^2 - D(N) N \log N \\
& \qquad + \frac{2}{1 - \frac{c^2}{16}} \Bigg( N I_0(\sigma) - \frac{c^2}{32} \log  N +  \frac{c^2}{16} \log \frac{c}{2} - \frac{c^2}{16 N} E_{0}(\omega_{N,2}) \Bigg) \\
& = I_0(\sigma) (N+1)^2  - D(N) N \log N - I_0(\sigma) \\
& \qquad + \frac{2}{1 - \frac{c^2}{16}} \Bigg( \frac{c^2}{16} D(N) \log N - \frac{c^2}{32} \log N  + O(1) \Bigg)
\end{align*}

Thus
\begin{align*}
D_1(& N+1) \geq D(N+1) \\
& \geq \frac{D(N) N \log N - \frac{2c^2}{16-c^2} D(N) \log N + \frac{c^2}{16-c^2} \log N + O(1) }{(N+1) \log(N+1)} \\
& = \frac{D(N)(N+1)\log(N+1) + D(N)\big[N \log\frac{N}{N+1} - \log(N+1) - \frac{2c^2}{16-c^2}  \log N\big] + \frac{c^2}{16-c^2} \log N + O(1)}{(N+1) \log(N+1)} \\
& = D(N) (1- \alpha(N)) + \frac{c^2}{16-c^2}  \frac{\log N}{(N+1)\log(N+1)} - \frac{O(1)}{(N+1)\log(N+1)}
\end{align*}
where $\alpha(N) \geq 0$ is $\Theta (\frac{1}{N})$.

Since we have $D(N+1) \geq D(N) + \frac{D(N)[2\log\frac{1}{2} - \frac{16+c^{2}}{16-c^{2}}\log(N+1)] + \frac{c^{2}}{16 - c^{2}} \log N + O(1)}{(N+1)\log(N+1)}$, the second term on the right becomes positive as soon as 
\begin{equation} \label{eq:DNineq}
	D(N) \leq \frac{\frac{c^{2}}{16 - c^{2}} \log N + O(1)}{-2\log\frac{1}{2} + \frac{16+c^{2}}{16-c^{2}}\log(N+1)} \rightarrow \frac{c^2}{16+c^2},
\end{equation}
Thus, if $D(N)$ becomes small enough, say smaller than $\frac{c^2}{2(16+c^2)}$ with $N$ sufficiently large, then $D(N+1) \geq D(N)$. Since we have $D(N) \leq D_{1}(N) \leq  \frac{1}{2} + \frac{C_{\log}}{\log N}$, and at each step we are adding the term of order $O(\frac{1}{N})$, the sequence might start decreasing again (but always staying larger than $\frac{c^2}{2(16+c^2)} - O(1/N)$). Thus, we can conclude that for all $N$ large enough, $\frac{c^2}{4(16+c^2)} \leq D(N) \leq \frac{1}{2} + \frac{C_{\log}}{\log N}.$

\qed

\section{Greedy Green Energy and Polarization} \label{sec:Green energy}

We begin by proving the well-separation of the Green greedy points on a manifold $\mathcal{M}$. The argument follows the same lines as in \cite[Theorem 1.1]{Del19}, where the author first proves the separation property of a minimizing configuration for the Green energy and then demonstrates that every harmonic ball $\mathcal{M}$ contains a geodesic ball of proportional radius. Harmonic balls on a manifold can be defined in two ways: either by the mean-value property for harmonic functions or via partial balayage (see \cite[Section 9]{GusR18}). In what follows, we shall use the latter definition. 

\begin{definition}
	Let $a$ be a real number with $0< a< V$, where $V = \textnormal{vol}(\mathcal{M})$ is the volume of $\mathcal{M}$, and denote by $G^{\delta_{p}}(x) = G(p,x)$ the Green's function centered at $p$. Consider the set
	$$
	K_{a} = \{u \in W^{1,2}(\mathcal{M}): u \leq a G^{\delta_{p}}\},
	$$
	where $W^{1,2}(\mathcal{M})$ denotes the Sobolev space of square integrable functions on $\mathcal{M}$ whose first weak derivatives are also square integrable. The functional 
	$$
	J_{a}(u) = \int_{\mathcal{M}} \langle \bigtriangledown u , \bigtriangledown u\rangle - 2 (1- a V^{-1}) u
	$$
	has a unique minimizer in $K_{a}$ which we denote by $u_{a}$. The \textit{harmonic ball} with center $p$ and volume $a$ is defined as a set
	$$
	B^{h} (p,a) = \{x \in \mathcal{M} : u_{a}(x) < a G^{\delta_{p}}(x)\}.
	$$
	
\end{definition}

\begin{lemma}  (\cite[Lemma 2.8]{Del19}) \label{lem:GreenLemma}
	Let $q$ be a point different from $p$. Assume that there is a function $f$, continuous in a neighborhood $U$ of $q$, such that 
	
	(1) $f(q) = a G(p, q)$,
	
	(2) $f(x) \leq a G(p, x)$ for $x \in U$,
	
	(3) $\triangle f \leq 1 - a V^{-1}$ on $U$.
	
	Then $q \notin B^{h} (p,a)$.  
	
\end{lemma}

\begin{theorem}\label{thm:sepGreenPolarization}
Let $\mathcal{M}$ be a compact Riemannian manifold of dimension $
	d \geq 2$, and let $\omega_{N} = \{x_{1}, ..., x_{N}\}$ be a configuration of $N \geq 2$ points in $\mathcal{M}$. If $x^* \in \mathcal{M}$ satisfies
    $$ \sum_{j=1}^{N} G(x^{*},x_j) = P_G(\omega_N)$$
then for every $j \in \{ 1, ..., N\}$,
	$$
	x^* \notin B^{h} \Big(x_{j}, \frac{V}{N}\Big),
	$$
	where $V = \vol (\mathcal{M})$ is the volume of $\M$.
\end{theorem}

\begin{proof} 
Let $N \geq 2$ and $j \in \{1,2,..., N\}$, and consider the function
	\begin{align*}
	h(x) &= \frac{V}{N} \bigg( G^{\delta_{x_{j}}}(x) + \sum_{\substack{i=1 \\ i\neq j}}^{N} G(x_{i}, x) \bigg)  \\ 
	     & = \frac{V}{2N} \bigg(  E_{G} (x_{1}, ...,x_{N}, x) - 2\sum_{1 \leq i < k \leq N} G(x_i, x_k)\bigg).
	\end{align*}
	Since the point $x^*$ is a global minimum for $h$, $x^* \not\in \{x_1,..., x_{N}\}$, so if we set
	$$
	f(x) = h(x^*) - \frac{V}{N}  \sum_{\substack{i=1 \\ i\neq j}}^{N} G(x_{i}, x),
	$$
	then $f$ is continuous in a neighborhood of $x^*$ and 
	$$
	\frac{V}{N} G^{x_{j}} (x) \geq f(x)
	$$
	for every $x \in \M$, with equality if $x = x^*$. We also have 
	$$
	\triangle f = -\frac{V}{N} \sum_{\substack{i=1 \\ i\neq j}}^{N} \triangle G^{\delta_{x_i}} \leq 
	\frac{N-1}{N} \vol = \bigg( 1- \frac{1}{N} \bigg) \vol.
	$$
	The proof is complete by taking $a = \frac{V}{N}$, $p = x_{j}$, and $q = x^*$ in Lemma \ref{lem:GreenLemma}.
	 
\end{proof}

Combining \cite[Theorem 1.4]{Del19}, which shows that every harmonic ball contains a geodesic ball of proportional radius, with Theorem \ref{thm:sepGreenPolarization}, and denoting by $d_{\mathcal{M}}(x,y)$ the geodesic distance between $x,y \in \mathcal{M}$, we immediately obtain the following corollary:

\begin{corollary}\label{cor:sepGreenPolarization}
Let $\mathcal{M}$ be a compact connected Riemannian manifold of dimension $d \geq 2$, and let $\omega_{N} = \{x_{1}, ..., x_{N}\}$ be a configuration of $N \geq 2$ points in $\mathcal{M}$. There is some constant $c>0$, independent of $N$ such that if $x^* \in \mathcal{M}$ satisfies
    $$ \sum_{j=1}^{N} G(x^*,x_j) = P_G(\omega_N),$$
then
	$$ d_{\mathcal{M}}(x_j, x^*) \geq c N^{-1/d}.
	$$
for every $j \in \{ 1, ..., N\}$.
\end{corollary} 

We can apply these results to greedily generated sequences.


\textit{Proof of Theorem \ref{thm:Separation bound for Greedy green points}}: Follows from Corollary \ref{cor:sepGreenPolarization}.

 \subsection{Green Polarization}\label{sec:GreenPolarization}

We define 
\begin{equation}
	I_G(\mu) : =  \int_{\mathbb{S}^d} \int_{\mathbb{S}^d} G(x, y) d\mu(x) d\mu(y).
\end{equation}

\begin{proof}[Proof of Theorem \ref{thm:GreenPolarizationAsymptotics}]

For each $N \in \mathbb{N}$, let $\{ x_{N,1}, ..., x_{N,N}\} = \omega_{N,d}$, and $x_{N, N+1} \in \mathbb{S}^d$ such that
\begin{equation}
\sum_{k=1}^{N} G( x_{N,k}, x_{N,N+1}) = P_G(\omega_{N,d}).
\end{equation}

Since our sequence of point configurations, $\{\omega_{N,d}\}$, is well-separated, as defined in Section \ref{subsec:Well-separated}, and due to Corollary \ref{cor:sepGreenPolarization}, we know that there is some constant $c > 0$ such that for each $N \in \mathbb{N}$
\begin{equation}
\min_{1 \leq i < j \leq N+1} \arccos( \langle x_{N,i},  x_{N,j} \rangle) \geq \arccos \bigg( 1 - \frac{c^2}{2 N^{2/d}} \bigg).
\end{equation}

As in (\ref{def:measure muN}), we define a probability measure $\mu_N$ with $r_N = \arccos \Big( 1 - \frac{c^2}{8 N^{2/d}} \Big) = 2 \arcsin\Big(\frac{c}{4 N^{1/d}} \Big)$, for each $N \in \mathbb{N}$. Then, using (\ref{eq:GreenExp}), (\ref{eq:GreenExp1}), and (\ref{eq:Green_at_opp_pts}), we obtain
\begin{align*}
	P_G( &\omega_N )  := \min_{x \in \mathbb{S}^d} \sum_{j=1}^{N} G(x, x_{N,j}) \\
	& \leq \sum_{j=1}^{N}  \int_{\mathbb{S}^d} G(x, x_{N,j}) d\mu_N (x) \\
	& = \frac{1}{1 - N \sigma ( B(p, r_N))} \Big( N I_G( \sigma) - N \int_{B(p, r_N)} G(x, p) d \sigma(x)  - 2\sum_{1 \leq i < j \leq N} \int_{B(x_{N,i}, r_N)} G(x, x_{N,j}) d \sigma(x) \Big)\\
	& = \frac{1}{1 - N \sigma ( B(p, r_N))} \Bigg( N I_G( \sigma) + N \sigma ( B(p,  r_N)) \frac{B_{\cos^2 \frac{r_N}{2}}(\frac{d}{2},\frac{d}{2})}{B_{\sin^2 \frac{r_N}{2}}(\frac{d}{2},\frac{d}{2})} \big(-K(\pi) + K (\pi-r_N)\big) \\
	&  \qquad - 2\sigma ( B(p,  r_N)) \sum_{1 \leq i < j \leq N} \big( G(x_{N,i},x_{N,j}) + K(r_N)\big) \Bigg),
\end{align*}
where $B_x(\alpha,\beta)$ is the incomplete beta function defined in (\ref{eqt:incomplete beta}). It is easy to see from (\ref{eqt:K2}) that the second term in the estimate is negative for sufficiently large $N$, so we replace it below with 0. By applying (\ref{eqt:K1}) to $K(r_N)$, and writing 
\begin{equation}\label{eq:D(N) for Green energy}
	E_G(\omega_{N,d}) =  - D(N) N^{2 - 2/d},
\end{equation}
we obtain 
 
\begin{align*}
P_G(\omega_N) & \leq \frac{\sigma ( B(p,  r_N))}{1 - N \sigma ( B(p, r_N))} \Big(  - N(N-1) \bigg(\frac{r_{N}^2}{(2d+4)V_d} + o(r_{N}^2) \bigg) - E_G(\omega_{N,d}) \Big)\\
& \leq \frac{\sigma ( B(p,  r_N))}{1 - N \sigma ( B(p, r_N))} \Big(  - N(N-1) \bigg(\frac{ c^2 N^{-2/d}}{4(2d+4)V_d} + o(N^{-2/d}) \bigg) + D(N) N^{2 - 2/d} \Big).
\end{align*}

Our claim now follows from the fact that $\sigma ( B(p,  r_N)) = \frac{A}{N}+ o(N^{-1})$ for some constant $A > 0$.
\end{proof}
 
\subsection{Asymptotics for Greedy Energy}\label{sec:Asymptotics for Green energy Greedy}

Let $\{ \omega_{N} \}$ be a sequence of greedily generated point configurations. As before, for each $N \geq 2$, we define $D(N)$ as in \eqref{eq:D(N) for Green energy} and $D_1(N)$ by
\begin{equation}
\mathcal{E}_G(N) = -  D_1(N) N^{2 - 2/d},
\end{equation} 
  so that $D(N) \leq D_{1}(N)$.
  
 Then 
 \begin{align*}
   -D(N+1) (N+1)^{2 - 2/d} &= E_G(\omega_{N+1}) \\
 & = E_G(\omega_{N}) + 2 	P_G(\omega_N)\\
 &\leq - D(N) N^{2 - 2/d} \\
 &+ \frac{2 \sigma ( B(p,  r_N))}{1 - N \sigma ( B(p, r_N))} \Big(    - N(N-1) \bigg(\frac{r_{N}^2}{(2d+4)V_d} + o(r_{N}^2) \bigg)  + D(N) N^{2 - 2/d} \Big)
 \end{align*}

 Finally, 
 \begin{align*}
 D(N+1) &\geq D(N) + \frac{D(N) [N^{2 - 2/d} -(N+1)^{2 - 2/d}  - \frac{2 \sigma ( B(p,  r_N)) N^{2 - 2/d}}{1 - N \sigma ( B(p, r_N))} ] }{(N+1)^{2 - 2/d}}\\
        & \qquad + \frac{ \frac{2 \sigma ( B(p,  r_N))}{1 - N \sigma ( B(p, r_N))} \Big[   N(N-1) \bigg(\frac{r_{N}^2}{(2d+4)V_d} + o(r_{N}^2) \bigg) \Big]}{(N+1)^{2 - 2/d}}
 \end{align*}

 The sum of the second and third terms is positive as soon as 
 
\begin{align*}
 D(N) \leq \frac{\frac{2 \sigma ( B(p,  r_N))}{1 - N \sigma ( B(p, r_N))} \Big[ N(N-1) \bigg(\frac{c^2 N^{-2/d}}{4(2d+4)V_d} + o(N^{-2/d}) \bigg) \Big]}{[(N+1)^{2 - 2/d} - N^{2 - 2/d}  + \frac{2 \sigma ( B(p,  r_N)) N^{2 - 2/d}}{1 - N \sigma ( B(p, r_N))} ]} \rightarrow  \frac{A c^2}{2(2d+4)V_d ((2 - 2/d)(1-A) + 2A)},
\end{align*} 
where $A$ is a constant for which $\sigma(B(p, r_N)) = \frac{A}{N} + \mathrm{o}(N^{-1})$  .
Thus, we conclude that the sequence $D(N)$ is bounded away from 0 for all $N$ sufficiently large.

\section{Separation Properties of Greedy Configurations for Riesz Kernels} \label{sec:separation}

Here we show bounds on separation for greedily generated point sets, by determining bounds on the distance between a point configuration, and the minimum of the potential they generate. In order to establish well-separation of the greedy configurations for Riesz $s$-kernel with $d-2 \leq s < d$, $d \geq 2$, we make use of the techniques developed in \cite{BraDS14, RakSZ} for separation of energy minimizing point configurations, in Section \ref{sec:Subharmonic separation}. For $s < d-2$, it is unknown if optimal point configurations for the energy are well-separated, and likewise, it is unclear if greedily generated sequences would be. We are able to provide a bound on separation of such point sequences in Section \ref{sec:Separation for s < d-2}, following a similar argument as that in \cite{DamM05}. 

\subsection{Separation for $d-2 \leq s < d$}\label{sec:Subharmonic separation}           
           
Let $\mathcal{P}(\S^{d})$ be a set of probability measures on $\S^d$. An \textit{external field} is defined as a lower semi-continuous function $Q: \S^d \ra (-\infty, \infty]$ with $Q(x) < \infty$ on a set of positive Lebesgue surface measure. The \textit{weighted energy associated with} $Q(x)$ is then given by 
\begin{equation}\label{Wenergy}
I(\mu) = I_{Q,s}(\mu) := \iint_{\S^d \times \S^d} (K_{s}(x-y) + Q(x) + Q(y)) d\mu(x) d\mu(y), \;\;\; \mu \in \mathcal{P}(\S^d).
\end{equation}
We note that semi-continuity of $Q$ implies the existence of a constant $c$ such that $Q(x) \geq c$ for all $x \in \S^d$. Therefore, the integrand in (\ref{Wenergy}) is bounded below and the double integral is well defined but possibly infinite.  

Following \cite{BraDS14}, we denote the Riesz $s$-potential of a measure $\mu \in \mathcal{P}(\S^d)$ by 
$$
U_{s}^{\mu}(x) :=\int K_s(x,y) d\mu(y),  \; x \in \R^{d+1},
$$
and the \textit{weighted s-potential} of the normalized counting measure $\mu_{\omega_N} = \frac{1}{N} \sum_{j=1}^{N} \delta_{x_{j}}$, associated with $N$-point configuration $\omega_N$, by

\begin{equation}\label{h-Def}
h_{\omega_N}(\mathbf{x}):=U_s^{\mu_{ \omega_N}}(\mathbf{x})+Q(\mathbf{x})=\frac{1}{N} \sum_{j=1}^N \frac{1}{s \|\mathbf{x}-\mathbf{x}_j\|^s}+Q(\mathbf{x}), \quad \mathbf{x} \in \mathbb{S}^d.
\end{equation}

The condition $-2 <s < d$ ensures the existence of the unique global minimizer $\mu_{Q,s} \in \mathcal{P}(\S^{d})$ of $I$, called the \textit{$s$-equilibrium measure associated with $Q$}:
$$
I(\mu_{Q,s}) = \underset{\mu \in \mathcal{P}(\S^{d})}{\min} I(\mu)
$$
The measure $\mu_{Q,s}$ is characterized by the variational Frostman conditions: there exists a constant $F_{Q}$ such that the modified potential 
\begin{equation} \label{Frostman}
\begin{split}
U_{s}^{\mu_{Q,s}} + Q  &\geq F_{Q},  \;\;\; \textnormal{q.e. on} \;\; \S^d\\
U_{s}^{\mu_{Q,s}} + Q  &\leq F_{Q},  \;\;\; \textnormal{on} \;\; \supp(\mu_{Q,s}),\\
\end{split}
\end{equation}
where q.e. is ``quasi–everywhere'' that denotes the property holding except on a set for which every probability measure supported on it has infinite energy.

\begin{theorem}(J. Brauchart, P. Dragnev, E. Saff, \cite[Theorem 12]{BraDS14})\label{BDS thm}
	
	Let $d \geq 2$, $d-2 \leq s < d$, $s>0$, and $\omega_{n} \subset \S^d$ be a set of $n$ points. Suppose that the associated weighted potential satisfies the inequality
	\begin{equation}\label{LBound:M}
	h_{\omega_n}(x) \geq M  \;\;\; \textnormal{q.e. on} \; \supp (\mu_{Q, s}),
	\end{equation}
	for some constant $M$. Then for all $x\in \S^d$ we have 
	\begin{equation}\label{BoundOnSphere}
	U_{s}^{\mu_{\omega_n}}(x) \geq M + U_{s}^{\mu_{Q,s}}(x) - F_{Q}.
	\end{equation}
	Furthermore,
	$$
		h_{\omega_n}(x) \geq M  \;\;\; \textnormal{q.e. on} \; \S^d.
	$$
\end{theorem}

Denote the {\it essential support} of $\mu_{Q,s}$ by 
$$
S_{Q}^{*} :=\{x \in \S^{d}: U_{s}^{\mu_{Q,s}}(x) + Q(x) \leq F_{Q}\} 
$$

\begin{theorem} \label{thm:pol_separation}
Let $d-2 \leq s < d$, $s >0$, and $\omega_{N} \subset \S^{d}$ be a set of $N$ points, $N>2$. Then the point $x^{*}$, such that
$$
P_{s}(\omega_{N}) = \sum_{j=1}^{N} K_{s}(x^{*}, x_{j}),   
$$
satisfies 
$$
\textnormal{dist} \;(x^{*}, x_{j}) \geq \frac{c}{N^{1/d}},\;\; j=1,2,...,N,
$$
with some constant $c>0$, independent of $N$, and our choice of $\omega_N$.
\end{theorem}

\begin{remark}
For $d=2$, $s=0$, the analogous result follows directly from Theorem 1 in \cite{RakSZ}. 
\end{remark}

\begin{proof}
First, observe that for each $i=1,2,3,...,N$ the point $x^{*}$ is a global minimum over $\S^{d}$ of the function $h_{E}$ given by (\ref{h-Def}) with $E=E_{i}=\omega_{N} \setminus \{{x_{i}}\}$ and $Q = Q_{i} = K_{s}(x, x_{i})/(N-1)$. Then, by taking $M=h_{E}(x^{*}) = U^{\mu_{E}}_{s}(x^{*}) + Q(x^{*})$, it follows from Theorem \ref{BDS thm} that
$$
U_{s}^{\mu_{E}}(x) \geq  U^{\mu_{E}}_{s}(x^{*}) + Q(x^{*}) +  U_{s}^{\mu_{Q,s}}(x) - F_{Q}, \;\; x \in \S^d.
$$
By taking $x=x^{*}$, we obtain $Q(x^{*}) +  U_{s}^{\mu_{Q,s}}(x^{*})  \leq F_{Q}$, which yields that $x^{*} \in S^{*}_{Q}$. In their work \cite[Propositions 14, 15]{BraDS14}, the authors show that the point $x^{*}$ in this case is separated from each $x_{i} \in \omega_{N}$, by a distance of at least $c_{s,d}N^{-1/d}$, with some $c_{s,d}>0$. 
\end{proof}

\noindent \textit{Proof of Theorem \ref{thm:sub_separation}:} Follows immediately from Theorem \ref{thm:pol_separation}.

\subsection{Separation for $0<s < d-2$}\label{sec:Separation for s < d-2}

\begin{theorem} \label{thm:pol_separation_small s}
Let $0 < s < d-2$, and $\omega_{N} \subset \S^{d}$ be a set of $N$ points, $N>2$. Then the point $x^{*}$, such that
$$
P_{s}(\omega_{N}) = \sum_{j=1}^{N} K_{s}(x^{*}, x_{j}),   
$$
satisfies 
$$
\textnormal{dist} \;(x^{*}, x_{j}) \geq \frac{c}{N^{1/(s+1)}},\;\; j=1,2,...,N,
$$
with some constant $c>0$, independent of $N$, and our choice of $\omega_N$.
\end{theorem}

\begin{proof}
For fixed $\varepsilon_0 >0$ sufficiently small, let $\varepsilon \in (0, \varepsilon_0)$. For any $x$ satisfying $1< \|x\| \leq 1 + \ve$, denote 
$\widetilde{x} := (2-\|x\|)\frac{x}{\|x\|}$. Then one can easily verify that 
$$
\underset{y \in \S^{d}}{\min} \frac{\|\widetilde{x} -y\|}{\|x-y\|} = 1 - \frac{2(\|x\|-1)}{\|x\|+1} > 1 - \ve,
$$

Then, for $1< \|x\| \leq 1 + \ve$, we have 
\begin{align*}
	U_{s}^{\nu^{*}_{s,N}} (x) &= N^{-1} \sum_{i=1}^{N} \frac{1}{s}\|x-x_{i}\|^{-s} \\&= N^{-1} \sum_{i=1}^{N} \frac{1}{s} \|\widetilde{x} - x_{i}\|^{-s} \bigg(\frac{\|\widetilde{x} - x_{i}\|}{\|x -x_{i}\|}\bigg)^{s} \geq (1 -\ve)^{s}   U_{s}^{\nu^{*}_{s,N}} (\widetilde{x}) 
\end{align*}
Set $\gamma_{N}:= \sum_{j=1}^{N} \frac{1}{s} \frac{1}{\|x^{*} - x_{j}\|^{s}} = \underset{y \in \S^{d}}{\inf} \sum_{j=1}^{N} \frac{1}{s}\frac{1}{\|y - x_{j}\|^{s}}$.

Since 
$
U_{s}^{\nu^{*}_{s,N}} (y)  \geq \frac{\gamma_{N}}{N},
$
for $y \in \S^d$, then by superharmonicity of the potential in $\mathbb{R}^{d+1}$ this inequality extends to $\|y\|\leq 1$,  and we have 
\begin{equation}
U_{s}^{\nu^{*}_{s,N}} (x)  \geq (1 -\ve)^{s}    \frac{\gamma_{N}}{N}
\end{equation}
	On the other hand, for the point $x$ lying on the ray through $0$ and $x^{*}$ with $\|x\| = 1+\ve$, we have $\|x-x_{i_{0}}\| \geq \ve$, $i_0 \in \{1,...,N\}$, and
\begin{align*} 
	U_{s}^{\nu^{*}_{s,N}} (x) &= N^{-1} \sum_{j \ne i_{0}} \frac{1}{s} \|x-x_{j}\|^{-s} + N^{-1} \frac{1}{s}\|x-x_{i_{0}}\|^{-s}
	\\  
	&\leq \frac{\gamma_{N}}{N} - N^{-1} \frac{1}{s} \|x^{*}-x_{i_{0}}\|^{-s} +N^{-1} \frac{1}{s} \ve^{-s}  
\end{align*} 
Combining the lower and upper bounds for the potential, we get
$$
\frac{\gamma_{N}}{N} - N^{-1}\frac{1}{s} \|x^{*}-x_{i_{0}}\|^{-s} +N^{-1} \frac{1}{s} \ve^{-s} \geq   (1-C\ve) \frac{\gamma_{N}}{N}, 
$$
By the rotational invariance of $\sigma$, it follows that $U_s^{\sigma}(x) = I_s$ for all $x \in \mathbb{S}^d$. Hence, we obtain the bound  
$$
\gamma_{N} \leq I_s N.
$$
Consequently,
$$
\frac{1}{s}\ve^{-s}  + C\ve I_s N \geq \frac{1}{s}\|x^{*}-x_{i_{0}}\|^{-s},
$$
so that
$$
\|x^{*}-x_{i_{0}}\| \geq C_{1} N^{-\frac{1}{s+1}}.
$$
\end{proof}

\noindent \textit{Proof of Theorem \ref{th:separation}:} Follows immediately from Theorem \ref{thm:pol_separation_small s}.

\subsection{Separation Properties Related to Polarization for $s \leq d-2$}\label{sec:Separation improvement, s < d-2}

\begin{theorem} \label{thm:improved separation superharmonic}
    Suppose $d \geq 3$, $s \in (0, d-2]$, and
    
\begin{equation}\label{eq:PolBound}
    P_s(\omega_N) \geq N I_s(\sigma) - Cs N^{\frac{s}{s+2}}
\end{equation}
for some $C > 0$. Then for any $x \in \mathbb{S}^d$ such that $ \sum_{j=1}^{N} K_s(x, x_j) = P_s(\omega_N)$, then for all $j \in \{ 1, ..., N\}$
\begin{equation}
    \|x - x_j\| \geq B N^{-\frac{1}{s+2}},
\end{equation}
where $B$ does not depend on $N$ or the specific choice of configuration satisfying (\ref{eq:PolBound}).
\end{theorem}

Here, we use a proof very similar to that of Theorem 3.5 in \cite{DamM05}.

\begin{proof}

Suppose $\omega_{N-1} = \{ x_1, ..., x_{N-1} \} \subset \mathbb{S}^d$ and set $x_N\in \mathbb{S}^d$ such that $\frac{1}{s} \sum_{j=1}^{N-1} \|x_N - x_j\|^{-s} = P_s(\omega_{N-1})$. Then for all $x \in \mathbb{S}^d$, we have that
\begin{equation}\label{eq:basic polarization bound}
    \frac{1}{N} \sum_{j=1}^{N-1} \frac{1}{s} \|x-x_j\|^{-s} \geq \frac{1}{N} P_s(\omega_{N-1}),
\end{equation}
so

\begin{align*}
    I_s(\sigma) & \geq  \frac{1}{N} \int_{ \mathbb{S}^d} \Bigg( \sum_{j=1}^{N-1} \frac{1}{s} \|x - x_j\|^{-s} \Bigg) d \sigma(x)\\
    & \geq \frac{1}{N} P_s(\omega_{N-1}).\\
\end{align*}

Due to the superharmonicity of $\frac{1}{N} \sum_{j=1}^{N-1} \frac{1}{s} \|x-x_j\|^{-s}$ on $\mathbb{R}^{d+1}$, by the minimal value principal, we also see that \eqref{eq:basic polarization bound} holds for $\|x\|\leq 1$.

 Just like in the beginning of the proof of Theorem \ref{thm:pol_separation_small s}, we can obtain the inequality

$$
    \frac{1}{N} \sum_{j=1}^{N-1} \frac{1}{s} \|x-x_j\|^{-s} \geq (1- \varepsilon)^s N^{-1} P_s( \omega_{N-1}).
$$

Now, let it be taken for granted that for some $\delta > 0$ and $j \in \{1, ..., N-1\}$, $\|x_N - x_j\| \leq \frac{\delta}{2}$. Then for any $x$ such that $\|x\| \leq 1 + \varepsilon$ and $\|x-x_N\| = \delta$,
$$ \frac{1}{N} \sum_{i =1}^{N-1} \frac{1}{s}\|x- x_i\|^{-s} - \frac{1}{s N} \|x- x_j\|^{-s} \geq (1- \varepsilon)^s N^{-1} P_s( \omega_{N-1}) - C_1 \frac{1}{N s}\delta^{-s}$$
for some constant $C_1 > 0$. Due to the superharmonicity of $\frac{1}{N} \sum_{i =1}^{N-1} \frac{1}{s}\|x- x_i\|^{-s} - \frac{1}{s N} \|x- x_j\|^{-s}$, on the hyperplane touching $\mathbb{S}^d$ at $x_N$, we have that
$$ \frac{1}{N} \sum_{i =1}^{N-1} \frac{1}{s}\|x_N - x_i\|^{-s} - \frac{1}{s N} \|x_N - x_j\|^{-s} \geq (1- \varepsilon)^s N^{-1} P_s( \omega_{N-1}) - C_1 \frac{1}{N s} \delta^{-s}.$$
Thus
\begin{equation}
    I_s(\sigma) \geq \frac{1}{N} \sum_{i =1}^{N-1} \frac{1}{s}\|x_N- x_i\|^{-s} \geq (1- \varepsilon)^s N^{-1} P_s( \omega_{N-1}) - C_1 \frac{1}{N s} \delta^{-s} + \frac{1}{s N} \|x_N- x_j\|^{-s}.
\end{equation}

Now, assume that for some $C_2 > 0$ and $a < 1$
\begin{equation}
    P_s( \omega_{N-1}) \geq N I_s(\sigma) - C_2 \frac{1}{s} N^a.
\end{equation}
Then
$$I_s(\sigma) \geq I_s(\sigma) - \frac{C_2}{s} N^{a-1} - C_3 s \varepsilon  - C_1 \frac{1}{N s} \delta^{-s} + \frac{1}{s N} \|x_N- x_j\|^{-s},$$
giving us
$$\frac{C_2}{s} N^{a} + C_3 s \varepsilon N + \frac{C_1}{s} \delta^{-s} \geq  \frac{1}{s} \|x_N- x_j\|^{-s}. $$

Again, considering the $d$-dimensional hyperplane touching $\mathbb{S}^d$ at $x_N$, then the points $x$ on the $(d-1)$-dimensional sphere on that hyperplane, with center $x_N$ and radius $\delta$, satisfy $\|x\| < 1 + \delta^2$. So setting $\delta = N^{-1/(s+2)}$ and $\varepsilon = \delta^2$, we have
$$\frac{C_2}{s} N^{a} + \frac{C_4}{s} N^{s/(s+2)} \geq  \frac{1}{s} \|x_N- x_j\|^{-s}.$$

Now, setting $a \geq \frac{s}{s+2}$, we have
$$\|x_N - x_j\| \geq C_5 N^{\frac{-1}{s+2}}.$$

\end{proof}

\section{Polarization Bounds for $s < d-2$} \label{sec:polarization for s<d-2}

The case $s<d-2$ is more complicated, primarily due to the lack of results concerning the well-separation of both optimal and greedy configurations. In this section we derive polarization bounds that are likely not optimal and could potentially be improved.

\begin{lemma} \label{lemma:meanv_subh}
Suppose $0<s < d-2$. If $y_0 \notin B(x_0,\ve)$, then
\begin{equation}
\int_{B(x_0,\ve)} K_{s}(x,y_0)d\sigma_d(x) \geq  \sigma(B(x_0,\ve)) K_{s}(x_0,y_0) 
+ C (s + 2 - 2d) \Big(1- \cos\Big(\frac{\theta_0}{2}\Big)\Big)^{-\frac{s}{2} - 1}  \ve^{d+2},
\end{equation}
where $\theta_0$ is an angle between $x_0$ and $y_0$.
\end{lemma}

\begin{proof}
 Assuming that the point $y_0$ is at the North pole, the kernel $K_{s}$ in polar coordinates is given by $K_{s}(x,y_0) = \frac{1}{s} (2 - 2 \cos(\theta))^{-s/2}$, and the Laplacian (taken with the opposite sign from that in the definition \eqref{eq:Laplace-Beltrami on Sphere}) is
\begin{equation} \label{laplacian}
\triangle_x f = 2^{-\frac{s}{2} - 1} (1- \cos\theta)^{-\frac{s}{2} - 1}\Big[\Big(\frac{s}{2} + 1 - d\Big) \cos \theta + \Big(\frac{s}{2} + 1\Big)\Big],
\end{equation}
where $\theta$ denotes an angle between $x$ and $y_0$.

Similar to \cite[Theorem A.2]{Bel13}, that follows from the proof of \cite[Prop 6.21]{Bes78}, the average of $K_{s}(x,y_0)$ over the boundary of the ball $B(x_0,r)$, not containing $y_0$, can be expressed as
\begin{equation} \label{average}
\frac{1}{\sigma_{d-1}(S(x_0,r))} \int_{S(x_0,r)} K_{s}(x,y_0) d \sigma_{d-1}(x) = K_{s}(x_0,y_0) + \int_{0}^{r} \frac{\int_{B(x_0,t)} \triangle_{x} K_{s}(x,y_0) d\sigma_d(x)}{\sigma_{d-1}(S(x_0,t))} dt
\end{equation}

Formula (\ref{average}) allows us to evaluate
\begin{align*}
\int_{B(x_0,\ve)} K_{s}(x,y_0) &d\sigma_d(x) = \int_{0}^{\ve} \int_{S(x_0, r)} K_{s} d\sigma_{d-1} dr\\
& = \int_{0}^{\ve} \sigma_{d-1}(S(x_0,r)) \bigg[ K_{s}(x_0,y_0) + \int_{0}^{r} \frac{\int_{B(x_0,t)} \triangle_{x} K_{s}(x,y_0) d\sigma_d(x)}{\sigma_{d-1}(S(x_0,t))} dt \bigg] dr\\
&= \sigma(B(x_0,\ve)) K_{s}(x_0,y_0) + \int_{0}^{\ve} \sigma_{d-1}(S(x_0,r))  \bigg[\int_{0}^{r} \frac{\int_{B(x_0,t)} \triangle_{x} K_{s}(x,y_0) d\sigma_d(x)}{\sigma_{d-1}(S(x_0,t))} dt \bigg]  dr.
\end{align*}

First, we investigate the integral $\int_{B(x_0,t)} \triangle_{x} K_{s}(x,y_0) d\sigma_d(x)$. Notice that since $s< d-2$, the expression in brackets in (\ref{laplacian}) becomes positive as soon as $\theta \in [\pi/2 - c,\pi]$, where $c = \arccos{\frac{s/2+1}{d-s/2-1}}$. In this case, we can assume without loss of generality that $\ve<c/2$, and whenever the ball $B(x_0, \ve)$ intersects the lower hemisphere, we take the lower bound for the integral over this ball to be 0. Thus, in the estimates below we will assume that $B(x_0, \ve)$ is located in the upper hemisphere, so that $\theta \in [0,\pi/2)$. Let $\theta_0$ denote the angle between $x_0$ and $y_0$. Notice that for all $x \in B(x_0,t)$ the angle $\theta$ between $x$ and $y_0$ satisfies $\cos(\theta) \leq \cos(\theta_0 - t)$.
Using this, we obtain
\begin{equation} \label{averLapl}
\int_{B(x_0,t)} \triangle_{x} K_{s}(x,y_0) d\sigma_d(x) \geq C_1 (\frac{s}{2} + 1 - d) (1- \cos(\theta_0 - t))^{-\frac{s}{2} - 1} \sigma_d(B(x_0, t)).
\end{equation}

Next, we use $\sigma_{d-1}(S(x_0,t)) = \frac{ 2\pi^{d/2}}{\Gamma(\frac{d}{2})} (\sin t)^{d-1}$ to get a bound $\int_{0}^{r} \frac{\sigma_d(B(x_0,t))}{\sigma_{d-1}(S(x_0,t))} dt \leq C_2 r^2$

Finally, combining the bounds, we get 

\begin{align*} \label{LowBound}
\int_{0}^{\ve} \sigma_{d-1}(S(x_0,r))  &\bigg[\int_{0}^{r} \frac{\int_{B(x_0,t)} \triangle_{x} K_{s}(x,y_0) d\sigma_d(x)}{\sigma_{d-1}(S(x_0,t))} dt \bigg]  dr \\ 
&\geq C_3 (\frac{s}{2} + 1 - d) (1- \cos(\theta_0 - \ve))^{-\frac{s}{2} - 1} \ve^{d+2}
\end{align*}

 Since $\ve  \leq \frac{\theta_0}{2} $, the bound above can be simplified to

 \begin{align*} 
\int_{0}^{\ve} \sigma_{d-1}(S(x_0,r))  &\Big[\int_{0}^{r} \frac{\int_{B(x_0,t)} \triangle_{x} K_{s}(x,y_0) d\sigma_d(x)}{\sigma_{d-1}(S(x_0,t))} dt \Big]  dr \\ 
&\geq C_4 \Big(\frac{s}{2} + 1 - d \Big) \Big(1- \cos\Big(\frac{\theta_0}{2}\Big) \Big)^{-\frac{s}{2} - 1}  \ve^{d+2}
\end{align*}

\end{proof}

\begin{theorem}
Suppose $0 < s < d-2$, and let a sequence of configurations $\{\omega_{N,d}\} \subset \S^{d}$ have a separation of order $N^{-1/\beta}$, $\beta \leq d$. Then there exist positive constants $c_{s,d}$ such that for all $N \in \mathbb{N}$, and all $\alpha \leq \frac{\beta(s+2)}{\beta+s+2}$,
\begin{equation}\label{eq:opt_polarization_subh}
  P_{s}(\omega_{N,d})   \leq I_{s,d} N - c_{s,d} N^{1+ \frac{s-d}{\alpha}}.
\end{equation}
If, in addition, $\beta > d-1$, then (\ref{eq:opt_polarization_subh}) holds for all $\alpha \leq \frac{\beta(s+2)}{(d-\beta)(s+1)+d}$.
\end{theorem}

\begin{proof}

Suppose that the sequence of configurations $\{\omega_{N,d}\}$ has a separation of order $N^{-1/\beta}$, $\beta \leq d$. That is, there is some constant $c>0$, such that for each $N \in \mathbb{N}$
\begin{equation}
\min_{1 \leq i < j \leq N} \|x_{N,i} - x_{N,j}\| \geq \frac{c}{N^{1/\beta}}. 
\end{equation}

This means that for $\alpha \leq \beta$,
\begin{equation}
\min_{1 \leq i < j \leq N} \arccos ( \langle x_{N,i},  x_{N,j} \rangle) \geq \arccos \bigg( 1 - \frac{c^2}{2 N^{2/\alpha}} \bigg).
\end{equation}

Similar to (\ref{def:measure muN}), we define a probability measure $\mu_N$ with $r_N = \arccos \Big( 1 - \frac{c^2}{8 N^{2/a}} \Big)$ for each $N \in \mathbb{N}$. Observe that, with this choice of radius, there exists a constant $A>0$, such that $\sigma(B(p, r_N)) = A N^{-\frac{d}{a}} + \mathrm{o}(N^{-\frac{d}{a}})$.

We have that
\begin{align*}
P_s(\omega_{N,d}) & := \min_{x \in \mathbb{S}^d} \sum_{j=1}^{N} K_s(x, x_{N,j}) \\
 & \leq \sum_{j=1}^{N}  \int_{\mathbb{S}^d} K_s(x, x_{N,j}) d\mu_N (x) \\
 & =  \frac{1}{1 - N \sigma ( B(p, r_N))} \Big( N I_s( \sigma) - N \int_{B(p, r_N)} K_s(x, p) d \sigma(x) \\
 &\qquad - 2\sum_{1 \leq i < j \leq N} \int_{B(x_{N,i}, r_N)} K_s(x, x_{N,j}) d \sigma(x) \Big).
\end{align*}
Using the fact that $\argmin_{B(p, r_N)} K_s(x, p)$ is on the boundary of the ball we have

\begin{equation} \label{eqt:bound_cent_ball}
\begin{split} 
 \int_{B(p, r_N)} K_s(x, p) d \sigma(x) & \geq \sigma(B(p, r_N)) \frac{1}{s}(2-2\cos(r_N))^{-s/2}\\
& = \sigma(B(p, r_N)) \frac{2^s}{s c^s} N^{s/a}
\end{split}
\end{equation}
Applying Lemma \ref{lemma:meanv_subh} yields the following bound for the sum
 \begin{equation} \label{eqt:estim on integral}
 \begin{split}
 - 2\sum_{1 \leq i < j \leq N} \int_{B(x_{N,i}, r_N)} K_s(x, x_{N,j}) d \sigma_d(x) & \leq
 - 2 \sigma_d(B(p,r_{N})) \sum_{1 \leq i < j \leq N}  K_s(x_{N,i}, x_{N,j})  \\
 &+ (d-1-\frac{s}{2}) C_5 N(N-1) (N^{-1/\beta})^{-s-2} (c N^{-1/\alpha})^{d+2}\\
 & \leq - \sigma_d(B(p,r_{N})) E_{s}(\omega_N) + C_6 c^{d+2} N^{2 + (s+2)/\beta - (d+2)/\alpha}
 \end{split}
 \end{equation}

 When $\beta > d-1$, a slightly stronger bound than that in (\ref{eqt:estim on integral}) can be established. 
 Namely, suppose that $N$ points are separated from each other by the distance $c_1N^{-1/\beta}$, and take geodesic balls to be of radius $c N^{-1/\alpha}$, $\alpha \leq \beta$. Let $x_{i_0} \in \omega_{N,d}$. Our goal is to improve the upper bound for the sum $\sum_{\substack{1\le j\le n\\ j\ne i_0}} \Big(1- \cos\Big(\frac{\theta_j}{2}\Big)\Big)^{-\frac{s}{2} - 1}$ for each $i_{0} \in \{1, ..., N\}$, where $\theta_j$ denotes the angle between $x_{N,i_0}$ and $x_{N,j}$. Without loss of generality, we can assume that $x_{i_0}$ is located at the North Pole. Since the closer a point $x_{N,j}$ is to $x_{N,i_0}$, the larger the term $\Big(1- \cos\Big(\frac{\theta_j}{2}\Big)\Big)^{-\frac{s}{2} - 1}$ becomes, it suffices to estimate the sum under the assumption that the remaining $N-1$ points of the configuration are positioned as close as possible to $x_{N,i_0}$, subject to the separation condition. Under this assumption, we partition the sphere into horizontal strips and estimate both the maximal number of balls in each strip and the total number of strips that these balls can occupy. To this end, divide the range of the polar angle $[0,\pi]$ into intervals of equal length $c_1N^{-1/\beta}$. The corresponding strip on the sphere with polar angle $\theta \in [kc_1N^{-1/\beta}, (k+1)c_1N^{-1/\beta}]$, $k =0,1,2...$, contains at most $O (N^{(d-1)/\beta})$ balls. If the number of balls in each strip is in fact $\Theta (N^{(d-1)/\beta})$, then these balls occupy only the top $\Theta(N^{(\beta - d +1)/\beta})$ strips on the sphere. Consequently, we obtain the bound
 
 $$
 \sum_{\substack{1\le j\le n\\ j\ne i_0}}\Big(1- \cos\Big(\frac{\theta_j}{2}\Big)\Big)^{-\frac{s}{2} - 1} \leq \sum_{k=1}^{N^{(\beta - d +1)/\beta}}  N^{(d-1)/\beta} [kN^{-1/\beta}]^{-s-2}.
 $$
 
 We then have
 
 \begin{equation} \label{eqt:estim on integral2}
 \begin{split}
 - 2\sum_{1 \leq i < j \leq N} \int_{B(x_{N,i}, r_N)} K_s(x, x_{N,j}) & d \sigma_d(x)  \leq 
 - 2 \sigma_d(B(p,r_{N})) \sum_{1 \leq i < j \leq N}  K_s(x_{N.i}, x_{N,j})  \\
 & + C_7 N \sum_{k=1}^{N^{(\beta - d +1)/\beta}}  N^{(d-1)/\beta} [kN^{-1/\beta}]^{-s-2}   (c N^{-1/\alpha})^{d+2} \\
 & \leq - 2 \sigma_d(B(p,r_{N})) \sum_{1 \leq i < j \leq N}  K_s(x_{N,i}, x_{N,j}) + C_8 c^{d+2} N^{\frac{2d -\beta s + d s }{\beta} - \frac{d+2}{\alpha}}.
 \end{split}
 \end{equation}

Combining the bounds (\ref{eqt:bound_cent_ball}), (\ref{eqt:estim on integral}), and applying Theorem \ref{thm:Riesz Energy Asymptotics}, it follows that
\begin{equation} \label{bound:polarization}  
\begin{split}
    P_s(\omega_N) 
    & \leq \frac{1}{1 - N \sigma ( B(p, r_N))} \Big( N I_s( \sigma) - N\sigma(B(p, r_N)) \frac{2^s}{s c^s} N^{s/a} - 2 \sigma_d(B(p,r_{N})) \sum_{1 \leq i < j \leq N}  K_s(x_{N,i}, x_{N,j})  \\
    & + C_6 c^{d+2} N^{2 + (s+2)/\beta - (d+2)/\alpha}\Big)\\
    & \leq \frac{1}{1 - N \sigma ( B(p, r_N))} \Big( N I_s( \sigma) - N\sigma(B(p, r_N)) \frac{2^s}{s c^s} N^{s/a} - \sigma_d(B(p,r_{N}))\big( N^2 I_s(\sigma) - C_s N^{1+s/d} \big)  \\
    & + C_6 c^{d+2}  N^{2 + (s+2)/\beta - (d+2)/\alpha}\Big)\\
    & = N I_s( \sigma) + \frac{1}{1 - N \sigma ( B(p, r_N))} \Big( C_9 c^{d} N^{1-d/\alpha}  ( - \frac{2^s}{s c^s} N^{s/a}  + C_s N^{s/d})  + C_6 c^{d+2}  N^{2 + (s+2)/\beta - (d+2)/\alpha}\Big)\\
    \end{split}
 \end{equation}   

Since $\alpha \leq d$, the term $ - \frac{2^s}{s c^s} N^{s/a}  + C_s N^{s/d}$ is negative provided that $c$ is taken sufficiently small. Thus, if $\alpha 
\leq \frac{\beta(s+2)}{\beta + s+  2}$, the bound (\ref{bound:polarization}) takes the form
\begin{equation} \label{eqt:polarizationBoundSubharmonic}
 P_s(\omega_N)  \leq N I_s( \sigma) - C_{10} N ^{1 + \frac{s-d}{\alpha}}.
\end{equation}

If $\beta > d-1$, then using the estimate (\ref{eqt:estim on integral2}) in place of (\ref{eqt:estim on integral}), yields a stronger version of bound (\ref{eqt:polarizationBoundSubharmonic}) which holds for all $\alpha \leq \frac{\beta (s+2)}{(d-\beta)(s+1) + d}$. In particular, for well-separated configurations the bound holds for all $\alpha \leq s+2$.

\end{proof}

\section{Acknowledgements}
Ryan W. Matzke was supported by the NSF Postdoctoral Research Fellowship Grant 2202877. The authors would like to thank Carlos Beltr\'{a}n, Michelle Mastrianni, and Stefan Steinerberger for their helpful conversations.

\section{Appendix A. Polarization and Greedy Sequences on $\mathbb{S}^1$} \label{sec:circle}

Maximal polarization and the energy of greedy points in the case $d = 1$  has mostly been settled in the literature previously, as the unique properties of the circle yield convenient examples of point sets that maximize polarization and greedy sequences. In this section we collect known results on the circle for polarization and greedy energy.

\subsection{Polarization on the Circle}
 
In \cite[Theorem 1]{HarKS13} it is shown that equally spaced points on $\mathbb{S}^1$, denoted $\omega_N^*$, are optimal for polarization for any kernel  $K(x,y) = f( \arccos(\langle x, y \rangle))$ such that $f$ is decreasing and convex on $[0, \pi]$. This includes the Riesz kernels for $-1 \le s < \infty$, but not $s < -1$. However, for $-2 < s < -1$, we conjecture that $\omega_N^*$ maximizes polarization as well. 

For the range $-2 < s$, the points that minimize the discrete potential with respect to $\omega_N^*$ are the midpoints of the arcs joining adjacent $N$th roots of unity. Because these midpoints are themselves $2N$th roots of unity, there is a natural expression for the polarization on $\mathbb{S}^1$ in terms of the corresponding energies for $N$ and $2N$ equally spaced points:
\begin{equation}
\label{eqn:diffofenergies}
P_{s}(\omega_N^*) = \frac{E_{s}(\omega_{2N}^*)}{2N} - \frac{E_{s}(\omega_N^*)}{N},
\end{equation}
see \cite{AmbBE13}. Furthermore, there is an explicit formula for the energy of $N$ equally spaced points on $\mathbb{S}^1$.

\begin{theorem}[Theorem 1.1 in \cite{BraHS09}] 
\label{thm:zetafunction}
Let $s \in \mathbb{R}$, $s \neq 0,1,3,5,...,$ and let $q$ be any nonnegative integer such that $q \ge (s -1)/2$. If $\omega_N^{\star}$ is a configuration of $N$ equally spaced points on $\mathbb{S}^1$, then
$$E_{s}(\omega_N^*) = I_{s,1} N^2 + \frac{2 }{s(2\pi)^s} \sum_{n=0}^q a_n(s) \zeta(s-2n)N^{1+s-2n}+\mathcal{O}(N^{s-1-2q}),$$
where $\zeta(s)$ is the classical Riemann zeta function and $a_n(s)$ are the coefficients in the expansion
$$\Big(\frac{\sin \pi z}{\pi z}\Big)^{-s} = \sum_{n=0}^{\infty} a_n(s) z^{2n}, \hspace{1cm} |z| < 1.$$
\end{theorem}

Combining Theorem \ref{thm:zetafunction} with \eqref{eqn:diffofenergies}, we have that on $\mathbb{S}^1$ for $-2 < s < 1$, $s \neq 0$, 
\begin{equation}
\label{eqn:expansion}
P_{s}(\omega_N^*) = I_{s,1}  N + \frac{2}{s(2\pi)^s} \zeta(s) N^s (2^s-1) + \mathcal{O}(N^{s-2}).
\end{equation}

This second order term was also explicitly computed in \cite[Proposition 4.1]{LopM22} and \cite[Lemma 3.10]{LopM21}. In \cite[page 623]{BraHS09}, the authors also show that 
\begin{equation}\label{eq:Log Energy equally spaced points}
E_{0}(\omega_N^*) = - N \log N
\end{equation}
which, when combined with \eqref{eqn:diffofenergies}, gives us that
\begin{equation}\label{eqn:expansion log Pol circle}
P_{0}(\omega_N^*) = - \log 2.
\end{equation}

Both \eqref{eqn:expansion} and \eqref{eqn:expansion log Pol circle} yield a lower bound on maximal polarization $\mathcal{P}_{s}(N) \ge P_{s}(\omega_N^*)$, which we conjecture to be an equality for $-2 < s < - 1$. This is true in the case $-1 \leq s$, where optimality of roots of unity, \cite[Theorem 1]{HarKS13} shows that $\mathcal{P}_{s}(N) = P_{s}(\omega_N^*)$.

\subsection{Energy for Greedy Sequences}\label{sec:greedyenergyS1}

The behavior of greedy sequences on $\mathbb{S}^1$ is also well-studied. For $-2 < s$, it is known that any greedy sequence is in fact a classical van der Corput sequence \cite{vdC35} (see \cite[Theorem 5]{BiaCC12}, \cite[Lemma 3.7]{LopM21}, \cite[Sec 1.2]{LopW15}, \cite[Lemmas 4.1 and 4.2]{LopS10}, and also \cite[Example 2]{Wol97}, which is perhaps the earliest observation of this kind, but just for $s=-1$). A similar result was shown for a large class of kernels  (whenever $K(x,y) = f( \arccos( \langle x , y \rangle))$, and $f$ is a bounded, continuous, decreasing, convex function of the geodesic distance) in \cite[Thm 2.1]{Pau21}.  The above mentioned results make explicit computations for bounds of the asymptotic behavior of the greedy Riesz energies possible. Here we collect these results.
\begin{theorem}\label{thm:Energy of Greedy points Circle}
On the circle $\mathbb{S}^1$, for $-2 < s < 1$ and $N \geq 2$,  if $\omega_N$ is the first $N$ elements of a greedy Riesz $s$-energy sequence, then there exist positive constants $C_{s,1}, C_{s,1}'$ such that the following hold for $N \geq 2$. 

\noindent For $0 < s < 1$, 
\begin{equation}\label{eq:Greedy Energy Circle 0 < s < 1}
- C_{s,1} N^{s+1} \leq E_{s}(\omega_N) - I_{s,1} N^2 \leq - C_{s,1}' N^{s+1}.
\end{equation}
For $s = 0$, 
\begin{equation}\label{eq:Greedy Energy Circle s=0}
0  \leq E_{0}(\omega_N) +  N \log N \leq \log(4/3) N.
\end{equation}
For $-1 < s < 0$,
\begin{equation}\label{eq:Greedy Energy Circle -1 < s < 0}
C_{s,1} N^{s+1} \leq E_{s}(\omega_N) - I_{s,1} N^2 \leq C_{s,1}' N^{s+1}.
\end{equation}
For $s = -1$
\begin{equation}\label{eq:Greedy Energy Circle s=-1}
C_{1,1}  \leq E_{1}(\omega_N) - I_{1,1} N^2 \leq C_{1,1}' \log N.
\end{equation}
For $-2 < s < -1$
\begin{equation}\label{eq:Greedy Energy Circle -2 < s < -1}
C_{s,1} N^{s+1} \leq E_{s}(\omega_N) - I_{s,1} N^2 \leq  C_{s,1}'.
\end{equation}
In each case, the order of the bounds cannot be improved.
\end{theorem}
The bounds for \eqref{eq:Greedy Energy Circle 0 < s < 1} and \eqref{eq:Greedy Energy Circle s=0} are provided in  \cite[Theorems 1.1, 1.2, 1.5]{LopW15}. The bounds for \eqref{eq:Greedy Energy Circle -1 < s < 0} were shown in \cite[Theorem 3.16]{LopM21}. The upper bounds of \eqref{eq:Greedy Energy Circle -2 < s < -1} and \eqref{eq:Greedy Energy Circle s=-1} are shown in \cite[Theorems 3.17, 3.18]{LopM21}. In each of the five cases, the lower bounds follow from the fact that $\mathcal{E}_{s}(N) \leq E_{s}(\omega_N)$, and the value of the minimal energy is known 
\begin{theorem}[Theorem 1.2 in \cite{CohK07}]
If $-2 < s < 1$, then on $\mathbb{S}^1$, $\mathcal{E}_s(N) = E_s(\omega_N)$ if and only if $\omega_N$ is a set of equally spaced points.  
\end{theorem}
Combining this with Theorem \ref{thm:zetafunction} yields the lower bounds in Theorem \ref{thm:Energy of Greedy points Circle}.

\subsection{Polarization for Greedy Sequences}

As discussed earlier, the fact that any greedy sequence is uniformly distributed suggests that they may behave well for different measures of uniformity. For polarization, the asymptotic behavior of a greedy sequence for Riesz kernels on the circle was shown in \cite[Theorem 3.11]{LopM21} and \cite[Theorems 1.1 and 1.4]{LopM22} (the case of hypersingular Riesz kernels was also handled in \cite{LopM22}  
{ and the case of $s \leq -2$ in \cite[Theorems 4.1 and 5.2]{LopM21}}).

\begin{theorem}\label{thm:Polar of Greedy points Circle}
On the circle $\mathbb{S}^1$, for $-2 < s < 1$ and $N \geq 1$, if $\omega_N$ is the first $N$ elements of a greedy sequence, then there exist positive constants $d_{s,1}, d_{s,1}'$, such that\\
For $0 < s < 1$,
\begin{equation}\label{eq:Polarization Circle 0 < s < 1}
- d_{s,1} N^{s} \leq P_{s}(\omega_{N}) - I_{s,1} N \leq - d_{s,1}' N^{s}.
\end{equation}
For $s = 0$, 
\begin{equation}\label{eq:Polarization Circle s=0}
-  \log(N+1) \leq P_{0}(\omega_{N}) - I_{0,1} N \leq 0.
\end{equation}
For $-2 < s < 0$,
\begin{equation}\label{eq:Polarization Circle -1 < s < 0}
- I_{s,1} \leq P_{s}(\omega_{N}) - I_{s,1} N \leq 0.
\end{equation}

\end{theorem}

Comparing this to (\ref{eqn:expansion}) and (\ref{eqn:expansion log Pol circle}), we see that the greedy sequences on the circle have good second-order asymptotic behavior for $0 < s$, but not for $-2 < s \leq 0$.

\section{Appendix B}

\subsection{Green Energies on Geodesic Balls}
In this section, we summarize the results of Beltr\'{a}n and Lizarte \cite{BelL23} on the mean value of the logarithmic and Green energies over geodesic balls.
First, recall that the geodesic ball in $\S^d$ centered at $p_0 \in \S^d$ with radius $a>0$ is denoted by
$B(p_0,a) =\{p \in \S^d : d_{R}(p_0, p) < a\} \subset \S^d$ where $d_{R}(\cdot, \cdot) = \arccos \langle\cdot, \cdot\rangle$ is the Riemannian distance on the unit sphere. By $\sigma(B(p_0,a))$ we denote the volume of this ball.

\begin{theorem}\label{thm:meanLog}
	Let $p_0, p \in \S^2$ and $a \in (0, \pi)$. Then, $\sigma(B(p_0,a)) = 4 \pi \sin^2 \frac{a}{2}$ and
	
	(i) If $p \notin B(p_0, a)$
	
	\begin{equation}
		\frac{1}{\sigma(B(p_0,a))} \int_{q \in B(p_0,a)} \log \| p - q\|dq = \log \|p - p_0\| - \frac{1}{2} - \cot^2 \frac{a}{2} \; \log \cos\frac{a}{2}.
	\end{equation} 
	
	(ii) If $p \in B(p_0, a)$
	
	\begin{align*}
		\frac{1}{\sigma(B(p_0,a))} \int_{q \in B(p_0,a)} \log \| p - q\|dq &= \log 2 - \frac{1}{2} \\
		& - \frac{\cot^2 \frac{a}{2}}{2} \log \bigg(1 - \frac{\|p-p_0\|^2}{4} \bigg) + \log \sin \frac{a}{2}.\\		
	\end{align*} 
	
\end{theorem}

\begin{theorem} \label{thm:Green_formula}
The Green function for $\S^d$ is $G(\S^d; p, q) = g(\|p-q\|)$, where
\begin{equation}
g(t) = \frac{2}{dV_d} \sum_{k=0}^{\infty} \frac{(d)_k}{(\frac{d}{2}+1)_k(k+1)}\bigg[\bigg(1 - \frac{t^2}{4}\bigg)^{k+1}  -\frac{B(\frac{d}{2},\frac{d}{2}+k+1)}{B(\frac{d}{2},\frac{d}{2})}\bigg],
\end{equation}
with $(d)_k$ denoting the Pochhammer symbol.
\end{theorem}

  In the mean value formula for the Green energy on $S^d$ the role of the term ``$-1/2-\cot^2(a/2) \log \cos(a/2)$'' plays the following expression:
 $$
K(a)=K(\mathbb{S}^{d},a):= \frac{2}{dV_d} \sum_{k=0}^{\infty} \frac{(d)_k}{(\frac{d}{2}+1)_k(k+1)} \frac{B_{\sin^2 \frac{a}{2}}(\frac{d}{2}+k+1,\frac{d}{2})}{B_{\sin^2 \frac{a}{2}}(\frac{d}{2},\frac{d}{2})},\; \;\;\; a\in(0,\pi),
$$
where the constant $V_d$ is given by
$$
V_d:=\textnormal{Vol}(\mathbb{S}^d) = \frac{2\pi^{\frac{d+1}{2}}}{\Gamma(\frac{d+1}{2})},
$$
and 
\begin{equation}  \label{eqt:incomplete beta}
B_{x}(\alpha, \beta) = \int_{0}^{x} t^{\alpha-1} (1-t)^{\beta-1} dt = \frac{x^{\alpha}}{\alpha} {}_{2}F_{1}(\alpha, 1-\beta; \alpha+1; x)
\end{equation}
is the incomplete beta function.

Note that, from Theorem \ref{thm:Green_formula}, 
\begin{equation} \label{eq:Green_at_opp_pts}
G(\S^d; p, -p) = g(2) = - K(\S^d,\pi), \quad \forall p \in \S^d.
\end{equation}

The following lemma presents some basic properties of $K(\S^d, a)$.

\begin{lemma}\label{lem:K(a)Lemma}
	We have $K(\S^2, a) = -1/(2\pi) (-1/2 - \cot^2(a/2) \log \cos(a/2))$ for $a \in (0, \pi)$. Moreover,
	\begin{equation} \label{eqt:K1}
		K(a) = 	\frac{a^2}{(2d+4)V_d} + o(a^2)
	\end{equation}
	\begin{equation} \label{eqt:K2}
		K(\pi - a) = K(\pi) - \frac{a^2}{(2d-4)V_d} + o(a^2)
	\end{equation}
\end{lemma}

\begin{theorem}($\S^d$-analog of Theorem \ref{thm:meanLog})
	Let $p_0, p \in \S^d$ and $a \in (0, \pi)$. Then, $\sigma(B(p_0,a)) = 2^{d-1} V_{d-1} B_{\sin^2 \frac{a}{2}}(\frac{d}{2}, \frac{d}{2})$ and moreover 
	
	(i) If $p \notin B(p_0, a)$:
	
	\begin{equation} \label{eq:GreenExp}
		\frac{1}{\sigma(B(p_0,a))} \int_{q \in B(p_0,a)} G(\S^d; p, q) dq = G(\S^d; p, p_0) + K(\S^d,a). 
	\end{equation}
	
	(ii) If $p \in B(p_0, a)$:
	
	\begin{equation} \label{eq:GreenExp1}
    \begin{split}
		\frac{1}{\sigma(B(p_0,a))} \int_{q \in B(p_0,a)} G(\S^d; p, q) dq &=  - \frac{B_{\cos^2 \frac{a}{2}}(\frac{d}{2}, \frac{d}{2})}{B_{\sin^2 \frac{a}{2}}(\frac{d}{2}, \frac{d}{2})} \\
		& \times (G(\S^d; p , -p_0) + K(\S^d, \pi-a)).
    \end{split}
	\end{equation}
	
\end{theorem}

\subsection{Upper Bound on Minimal Green Energy}\label{sec:Upper bound on Green energy}

\begin{lemma}\label{lem:equal area partition}
  For each $d \in \mathbb{N}$, there is a positive constant $c_d$ such that for all $N \in \mathbb{N}$, there is a partition of $\mathbb{S}^d$ into $N$ regions, with each region having measure $\frac{1}{N}$ and diameter at most $c_d N^{-1/d}$.
\end{lemma}
Lemma \ref{lem:equal area partition} was treated in the 1987 book of Beck and Chen \cite{BeckC87} as a ``well-known result". An explicit construction for $d=2$ and all $N$ was likely first given in the 1994 paper by Rakhmanov, Saff, and Zhou \cite{RakSZ94}, and was later extended by Leopardi \cite{Leo09} for all $d \geq 2$ using an iterative construction on $d$.

Combining Proposition 3.1 and Lemma C.2 in \cite{BelL23}, we have that for any $x,y \in \mathbb{S}^d$, $d \geq 3$
\begin{equation}\label{eq:Green Function Expansion}
G(x,y) = \frac{\Gamma(\frac{d+1}{2})^2 \Gamma(\frac{d-1}{2})}{d! \pi^{\frac{d+1}{2}}} \|x-y\|^{2-d}+ \mathrm{o}(\|x-y\|^{2-d}).
\end{equation}\label{eq:asymptotics of Green kernel}


\begin{proposition}
  For $\mathbb{S}^d$, $d \geq 3$, there is some positive constant $c_G$ such that for $N \in \mathbb{N}$
  sufficiently large,
\begin{equation*}
\mathcal{E}_{G}(N)  \leq -c_G N^{2- \frac{2}{d}}
\end{equation*}
\end{proposition}

\begin{proof}
   By Lemma \ref{lem:equal area partition}, there exists some positive constant $c_{d}$ such that for $N$
  sufficiently large, there is a partition of $\mathbb{S}^d$ into $N$ regions, $D_1$,
  \ldots, $D_N$, each of measure $\frac{1}{N}$ and with diameter at most $c_d N^{-1/d}$.

  Letting $ d \sigma_j(x) := N \mathbf{1}_{D_j} d \sigma(x)$, we have, from \eqref{eq:Green Function Expansion},
\begin{align*}
\mathcal{E}_{G}(N) &  \leq \int_{\Omega} \cdots \int_{\Omega} \sum_{i \neq j} G( z_i, z_j)\,d \sigma_1(z_1) \cdots d \sigma_N(z_N) \\
& = N^2 \sum_{i \neq j} ~\iint\limits_{D_i\times D_j} G( z_i, z_j) \,d \sigma(z_i)\, d \sigma(z_j) \\
& = N^2 \Bigg(~\iint\limits_{\mathbb{S}^d\times\mathbb{S}^d} G (x,y) \,d \sigma(x)\, d \sigma(y) - \sum_{j=1}^{N} ~\iint\limits_{D_j\times D_j} G( x, y) \,d \sigma_j(x)\,
d \sigma_j(y)\Bigg)\\
& \leq N^2 I_{G}(\sigma) - \sum_{j=1}^{N} \left( \frac{\Gamma(\frac{d+1}{2})^2 \Gamma(\frac{d-1}{2})}{d! \pi^{\frac{d+1}{2}}} \operatorname{diam}(D_j)^{2-d}  + \mathrm{o}\Big(\operatorname{diam}(D_j)^{2-d} \Big) \right)\\
& \leq  - \frac{ \Gamma(\frac{d+1}{2})^2 \Gamma(\frac{d-1}{2})}{d! \pi^{\frac{d+1}{2}}} \frac{1}{(c_d)^{d-2}} N^{2- \frac{2}{d}} + \mathrm{o}( N^{2-\frac{2}{d}}).
\end{align*}

\end{proof}

\bibliographystyle{alpha}

\end{document}